\documentclass[11pt, a4paper]{article}
\usepackage{mathrsfs, mathtools, amsthm, amssymb, thm-restate}  
\usepackage{paralist, tablists}  

\usepackage{graphicx, xcolor, tikz}  
\usetikzlibrary{calc, shapes, decorations.pathreplacing, decorations.markings}  
\usepackage{caption}  
\usepackage[labelformat=simple]{subcaption}  

\usepackage[left=25mm, top=25mm, bottom=25mm, right=25mm]{geometry}  
\usepackage[T1]{fontenc} 

\usepackage[inline]{enumitem}  
\usepackage[square, numbers, sort&compress]{natbib}  
\usepackage[
  bookmarks=true,                
  bookmarksnumbered=true,        
  bookmarksopen=true,            
  plainpages=false,              
  colorlinks=true,               
  linkcolor=blue,                
  citecolor=black,               
  anchorcolor=green,             
  urlcolor=blue,                 
  hyperindex=true                
]{hyperref} 

\newtheorem{theorem}{Theorem}[section] 
\newtheorem{lemma}[theorem]{Lemma}
\newtheorem{conjecture}[theorem]{Conjecture}
\newtheorem{corollary}[theorem]{Corollary}

\theoremstyle{definition}

\newtheorem{remark}{Remark}

\makeatletter
\def\th@plain{%
  \upshape 
}
\makeatother

\newcommand{\etal}{et~al.\ }
\newcommand{\ie}{i.e.,\ }

\usepackage{marvosym,wasysym}
\usepackage{array}
\usepackage{bm}  

\def\int(#1){\mathrm{int}(#1)}
\def\ext(#1){\mathrm{ext}(#1)}
\def\Int(#1){\mathrm{Int}(#1)}
\def\Ext(#1){\mathrm{Ext}(#1)}

\makeatletter
\renewenvironment{proof}[1][\proofname]{\par
  \pushQED{\qed}%
  \normalfont \topsep6\p@\@plus6\p@\relax
  \trivlist
  \item[\hskip\labelsep
        \bfseries
    #1\@addpunct{.}]\ignorespaces
}{%
  \popQED\endtrivlist\@endpefalse
}
\makeatother

\usepackage[capitalise]{cleveref}  
\crefname{claim}{Claim}{Claims}

\begin{document}

\title{Variable degeneracy of planar graphs without chorded 6-cycles}
\author{Huihui Fang\footnote{School of Mathematics and Statistics, Henan University, Kaifeng, 475004, China.} \and Danjun Huang\footnote{School of Mathematical Sciences, Zhejiang Normal University, Jinhua 321004, China.} \and Tao Wang\footnote{Center for Applied Mathematics, Henan University, Kaifeng, 475004, China. {\tt Corresponding author: wangtao@henu.edu.cn; https://orcid.org/0000-0001-9732-1617}} \and Weifan Wang\footnote{School of Mathematical Sciences, Zhejiang Normal University, Jinhua 321004, China.}}
\date{}
\maketitle
\begin{abstract}
A cover of a graph $G$ is a graph $H$ with vertex set $V(H) = \bigcup_{v \in V(G)} L_{v}$, where $L_{v} = \{v\} \times [s]$, and the edge set $M = \bigcup_{uv \in E(G)} M_{uv}$, where $M_{uv}$ is a matching between $L_{u}$ and $L_{v}$. A vertex set $T \subseteq V(H)$ is a transversal of $H$ if $|T \cap L_{v}| = 1$ for each $v \in V(G)$. Let $f$ be a nonnegative integer valued function on the vertex-set of $H$. If for any nonempty subgraph $\Gamma$ of $H[T]$, there exists a vertex $x \in V(H)$ such that $d(x) < f(x)$, then $T$ is called a strictly $f$-degenerate transversal. In this paper, we give a sufficient condition for the existence of strictly $f$-degenerate transversal for planar graphs without chorded $6$-cycles. As a consequence, every planar graph without subgraphs isomorphic to the configurations in \cref{fig:FORBID} is DP-$4$-colorable. 


\end{abstract}

\section{Introduction}

All graphs in this paper are finite, simple and undirected. A graph $G$ is planar if it has a drawing in the Euclidean plane without crossings. Such a drawing is a planar embedding of $G$. A plane graph is a particular planar embedding of a planar graph. A {\em $k$-cycle} is a cycle of length $k$, and a $3$-cycle is usually called as a triangle. Two cycles are {\em adjacent} (or {\em intersecting}) if they share at least one edge (or vertex). Two adjacent cycles are called to be {\em normally adjacent} if the intersection is a single edge and exactly two vertices.


List coloring is a well-known generalization of proper $k$-coloring, introduced by Vizing \cite{MR0498216} and independently by Erd\H{o}s, Rubin, and Taylor \cite{MR593902}. Let $L$ assign a list $L(v)$ of possible colors to each vertex $v$ of $G$ with $|L(v)| \geq k$. Then we say that $L$ is a $k$-list assignment for the graph $G$. If $G$ has a proper coloring $\phi$ such that $\phi(v) \in L(v)$ for each vertex $v$, then we say that $G$ is {\em $L$-colorable}. If $G$ is $L$-colorable for any $k$-list assignment $L$, then $G$ is {\em $k$-choosable}. The {\em list chromatic number} of $G$, denoted by $\chi_{\ell}(G)$, is the smallest positive integer $k$ such that $G$ is $k$-choosable. By the definition, it holds trivially that $\chi_{\ell}(G) \geq \chi(G)$ for any graph $G$, where $\chi(G)$ is the chromatic number of $G$. 

Thomassen \cite{MR1290638} showed that every planar graph is 5-choosable, and Voigt \cite{MR1235909} found an example that is not a $4$-choosable planar graph. It is interesting in graph coloring to find sufficient conditions for planar graphs to be $4$-choosable. Note that every triangle-free planar graph is $4$-choosable since it is $3$-degenerate. Moreover, Wang and Lih \cite{MR1935837} improved this result by proving that planar graphs without intersecting $3$-cycles are $4$-choosable. A planar graph is $4$-choosable if it does not have $k$-cycles for $k = 4$ \cite{MR1687318}, for $k = 5$ \cite{MR1889505}, for $k = 6$ \cite{MR1842116}, for $k = 7$ \cite{MR2538645}. 

In 2018, Dvo\v{r}\'{a}k and Postle \cite{MR3758240} introduced the concept of DP-coloring (under the name {\em correspondence coloring}), and showed that DP-coloring is a generalization of list coloring. 

Let $G$ be a graph and $L_{v} = \{v\} \times [s]$ for $v \in V(G)$, where $[s] = \{1, 2, \dots, s\}$. For each edge $uv$ in $G$, let $M_{uv}$ be a matching between the sets $L_{u}$ and $L_{v}$. Let $M = \cup_{uv \in E(G)}M_{uv}$, which is called a matching assignment. Then a graph $H$ is said to be the $M$-cover of $G$ if it satisfies all the following conditions:

(i) the vertex set of $H$ is $\bigcup_{u \in V(G)}L_{u}$;

(ii) the edge set of $H$ is $M$.

A {\em transversal} of an $M$-cover $H$ is a vertex subset $T$ of $V(H)$ with $|T \cap L_{v}| = 1$ for each $v \in V(G)$. Let $f$ be a function from $V(H)$ to $\{0, 1, 2, \dots\}$. A transversal $T$ is {\em strictly $f$-degenerate} if every nonempty subgraph of $H[T]$ has a vertex $x$ of degree less than $f(x)$ in this subgraph. When $f$ is restricted to be a function from $V(H)$ to $\{0, 1\}$, a strictly $f$-degenerate transversal of $H$ is called a DP-coloring of $H$. If there is no confusion, then a DP-coloring of $H$ is also called a DP-coloring of $G$. 

The {\em DP-chromatic number}, denoted by $\chi_{DP}(G)$, is the minimum integer $k$ such that $H$ admits a DP-coloring whenever $H$ is a cover of $G$ and $f$ is a function from $V(H)$ to $\{0, 1\}$ with $\sum_{x \in L_{v}}f(x) \geq k$ for each $v \in V(G)$. We say that a graph $G$ is {\em DP-$k$-colorable} if $\chi_{DP}(G) \leq k$.

As in proper coloring, we call the elements of $[s]$ as colors, and call the element $i$ as the color of $v$ if $(v, i)$ is chosen in a transversal of $H$. And a vertex $v$ colored with $i$ is also stated that $i$ is the color of $v$. Note that DP-coloring is a generalization of list coloring. This implies that $\chi_{\ell}(G) \leq \chi_{DP}(G)$. It is obvious that DP-coloring and list coloring can be quite different. For example, $\chi_{\ell}(C) < \chi_{DP}(C)$ for every even cycle $C$.

Dvo\v{r}\'{a}k and Postle \cite{MR3758240} followed Thomassen's proof \cite{MR1290638} to show that every planar graph $G$ has $\chi_{DP}(G) \leq 5$. Kim and Ozeki \cite{MR3802151} proved that for each $k \in \{3, 4, 5, 6\}$, every planar graph without $k$-cycles is DP-$4$-colorable. Recently, it is proved that every planar graph is DP-$4$-colorable if it does not contain $i$-cycles adjacent to $j$-cycles for distinct $i$ and $j$ from $\{3, 4, 5, 6\}$, see \cite{MR3969022,MR4078909,MR3996735,MR4654340}. More sufficient conditions for a planar graph to be DP-4-colorable see \cite{MR3881665,MR4294211,MR4089638,MR4212281}. 

For strictly $f$-degenerate transversals, we are interested in the value
\[
\eta(G) = \min_{v \in V(G)}\left\{\sum_{x \in L_{v}}f(x)\right\}
\]
which guarantees the existence of strictly $f$-degenerate transversals whenever $H$ is a cover of $G$ and $f$ is a function on $V(H)$. 

It is interesting to find sufficient conditions for small value of $\eta(G)$. Sittitrai and Nakprasit \cite{MR4345150} proved that every planar graph without $3$-cycles adjacent to $4$-cycles has a strictly $f$-degenerate transversal whenever $f$ is a mapping to $\{0, 1, 2\}$ and $\eta(G) \geq 4$. Nakprasit and Nakprasit \cite{MR4114324} proved that every planar graph has a strictly $f$-degenerate transversal whenever $f$ is a mapping to $\{0, 1, 2\}$ and $\eta(G) \geq 5$. They also proved that every planar graph without pairwise adjacent $3$-, $4$- and $5$-cycles has a strictly $f$-degenerate transversal whenever $f$ is a mapping to $\{0, 1, 2\}$ and $\eta(G) \geq 4$. The third result they obtained is that every planar graph without cycles of lengths $4, a, b, 9$, where $a$ and $b$ are distinct values from $\{6, 7, 8\}$, has a strictly $f$-degenerate transversal whenever $f$ is a mapping to $\{0, 1, 2\}$ and $\eta(G) \geq 3$. 

\begin{figure}[htbp]%
\centering
\subcaptionbox{\label{fig:subfig:a--}}
{\begin{tikzpicture}[scale = 0.8]
\coordinate (A) at (45:1);
\coordinate (B) at (135:1);
\coordinate (C) at (225:1);
\coordinate (D) at (-45:1);
\coordinate (H) at (90:1.414);
\draw (A)--(H)--(B)--(C)--(D)--cycle;
\draw (A)--(B);
\node[circle, inner sep = 1.5, fill = white, draw] () at (A) {};
\node[circle, inner sep = 1.5, fill = white, draw] () at (B) {};
\node[circle, inner sep = 1.5, fill = white, draw] () at (C) {};
\node[circle, inner sep = 1.5, fill = white, draw] () at (D) {};
\node[circle, inner sep = 1.5, fill = white, draw] () at (H) {};
\end{tikzpicture}}\hspace{1.5cm}
\subcaptionbox{\label{fig:subfig:b--}}
{\begin{tikzpicture}[scale = 0.8]
\coordinate (A) at (30:1);
\coordinate (B) at (150:1);
\coordinate (C) at (225:1);
\coordinate (D) at (-45:1);
\coordinate (H) at (90:1.414);
\coordinate (X) at (60:1.4);
\coordinate (Y) at (120:1.4);
\coordinate (T) at ($(H)!(A)!(X)$);
\coordinate (Z) at ($(A)!2!(T)$); 
\draw (A)--(X)--(Z)--(H)--(Y)--(B)--(C)--(D)--cycle;
\draw (A)--(H)--(B);
\draw (H)--(X);
\node[circle, inner sep = 1.5, fill = white, draw] () at (A) {};
\node[circle, inner sep = 1.5, fill = white, draw] () at (B) {};
\node[circle, inner sep = 1.5, fill = white, draw] () at (C) {};
\node[circle, inner sep = 1.5, fill = white, draw] () at (D) {};
\node[circle, inner sep = 1.5, fill = white, draw] () at (H) {};
\node[circle, inner sep = 1.5, fill = white, draw] () at (X) {};
\node[circle, inner sep = 1.5, fill = white, draw] () at (Y) {};
\node[circle, inner sep = 1.5, fill = white, draw] () at (Z) {};
\end{tikzpicture}}\hspace{1.5cm}
\subcaptionbox{\label{fig:subfig:c--}}
{\begin{tikzpicture}[scale = 0.8]
\coordinate (A) at (30:1);
\coordinate (B) at (150:1);
\coordinate (C) at (225:1);
\coordinate (D) at (-45:1);
\coordinate (H) at (90:1.414);
\coordinate (X) at (60:1.4);
\coordinate (Y) at (120:1.4);
\coordinate (T) at ($(A)!(H)!(X)$);
\coordinate (Z) at ($(H)!2!(T)$); 
\draw (A)--(Z)--(X)--(H)--(Y)--(B)--(C)--(D)--cycle;
\draw (X)--(A)--(H)--(B);
\node[circle, inner sep = 1.5, fill = white, draw] () at (A) {};
\node[circle, inner sep = 1.5, fill = white, draw] () at (B) {};
\node[circle, inner sep = 1.5, fill = white, draw] () at (C) {};
\node[circle, inner sep = 1.5, fill = white, draw] () at (D) {};
\node[circle, inner sep = 1.5, fill = white, draw] () at (H) {};
\node[circle, inner sep = 1.5, fill = white, draw] () at (X) {};
\node[circle, inner sep = 1.5, fill = white, draw] () at (Y) {};
\node[circle, inner sep = 1.5, fill = white, draw] () at (Z) {};
\end{tikzpicture}}
\caption{Forbidden configurations in \cref{MRTHREE}.}
\label{E}
\end{figure}
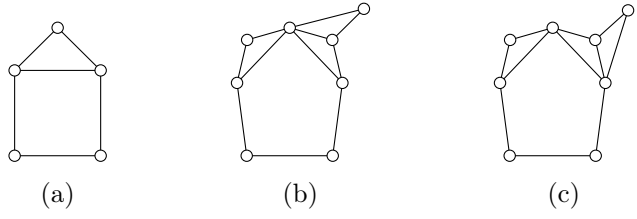

\begin{theorem}[Li and Wang \cite{MR4669976}]\label{MRTHREE}
Let $G$ be a planar graph without subgraphs isomorphic to the configurations in \cref{E}. Let $H$ be a cover of $G$ and $f$ be a function from $V(H)$ to $\{0, 1, 2\}$. If $f(v, 1) + f(v, 2) + \dots + f(v, s) \geq 4$ for each $v \in V(G)$, then $H$ has a strictly $f$-degenerate transversal. 
\end{theorem}

\begin{figure}[htbp]%
\centering
\subcaptionbox{\label{fig:subfig:a}}
{\begin{tikzpicture}[scale = 0.8]
\def\s{1.3}
\coordinate (A) at (\s, 0);
\coordinate (B) at (60:\s);
\coordinate (O) at (0, 0);
\coordinate (C) at (-150:\s);
\coordinate (D) at (0,-\s);
\coordinate (E) at (\s, -\s);
\draw (A)--(B)--(O)--(C)--(D)--(E)--cycle;
\draw (A)--(O)--(D);
\node[circle, inner sep = 1.5, fill = white, draw] () at (O) {};
\node[circle, inner sep = 1.5, fill = white, draw] () at (A) {};
\node[circle, inner sep = 1.5, fill = white, draw] () at (B) {};
\node[circle, inner sep = 1.5, fill = white, draw] () at (C) {};
\node[circle, inner sep = 1.5, fill = white, draw] () at (D) {};
\node[circle, inner sep = 1.5, fill = white, draw] () at (E) {};
\end{tikzpicture}}\hspace{1.5cm}
\subcaptionbox{\label{fig:subfig:b}}
{\begin{tikzpicture}[scale = 0.8]
\def\s{1}
\coordinate (A) at (\s, 0);
\coordinate (B) at (0.5*\s, 0.7*\s);
\coordinate (O) at (0, 0);
\coordinate (E) at (\s, -\s);
\coordinate (D) at (0,-\s);
\coordinate (C) at (0.5*\s, -1.7*\s);
\draw (A)--(B)--(O)--(D)--(C)--(E)--cycle;
\draw (A)--(O);
\draw (D)--(E);
\node[circle, inner sep = 1.5, fill = white, draw] () at (O) {};
\node[circle, inner sep = 1.5, fill = white, draw] () at (A) {};
\node[circle, inner sep = 1.5, fill = white, draw] () at (B) {};
\node[circle, inner sep = 1.5, fill = white, draw] () at (C) {};
\node[circle, inner sep = 1.5, fill = white, draw] () at (D) {};
\node[circle, inner sep = 1.5, fill = white, draw] () at (E) {};
\end{tikzpicture}}\hspace{1.5cm}
\subcaptionbox{\label{fig:subfig:433}}
{\begin{tikzpicture}[scale = 0.8]
\def\s{1.2}
\coordinate (O) at (0, 0);
\coordinate (A) at (0, -\s);
\coordinate (B) at (-\s,-\s);
\coordinate (C) at (-\s,0);
\coordinate (D) at (135:\s);
\coordinate (E) at (0,\s);
\draw (O)--(A)--(B)--(C)--(D)--(E)--cycle;
\draw (C)--(O)--(D);
\node[circle, inner sep = 1.5, fill = white, draw] () at (O) {};
\node[circle, inner sep = 1.5, fill = white, draw] () at (A) {};
\node[circle, inner sep = 1.5, fill = white, draw] () at (B) {};
\node[circle, inner sep = 1.5, fill = white, draw] () at (C) {};
\node[circle, inner sep = 1.5, fill = white, draw] () at (D) {};
\node[circle, inner sep = 1.5, fill = white, draw] () at (E) {};
\end{tikzpicture}}\hspace{1.5cm}
\subcaptionbox{\label{fig:subfig:3444}}
{\begin{tikzpicture}[scale = 0.8]
\def\s{1}
\coordinate (O) at (0, 0);
\coordinate (E) at (\s, 0);
\coordinate (W) at (-\s,0);
\coordinate (N) at (0,\s);
\coordinate (S) at (0,-\s);
\coordinate (NW) at (-\s,\s);
\coordinate (SW) at (-\s,-\s);
\coordinate (SE) at (\s,-\s);
\draw (E)--(N)--(NW)--(SW)--(SE)--cycle;
\draw (N)--(S);
\draw (E)--(W);
\node[circle, inner sep = 1.5, fill = white, draw] () at (O) {};
\node[circle, inner sep = 1.5, fill = white, draw] () at (E) {};
\node[circle, inner sep = 1.5, fill = white, draw] () at (W) {};
\node[circle, inner sep = 1.5, fill = white, draw] () at (N) {};
\node[circle, inner sep = 1.5, fill = white, draw] () at (S) {};
\node[circle, inner sep = 1.5, fill = white, draw] () at (SE) {};
\node[circle, inner sep = 1.5, fill = white, draw] () at (SW) {};
\node[circle, inner sep = 1.5, fill = white, draw] () at (NW) {};
\end{tikzpicture}}\\
\subcaptionbox{\label{fig:subfig:533}}
{\begin{tikzpicture}[scale = 0.8]
\def\s{1}
\foreach \ang in {1, 2, 3, 4, 5}
{
\def\pointname{v\ang}
\coordinate (\pointname) at ($(\ang*360/5-54:\s)$);}
\coordinate (S) at ($(v2)!1!60:(v1)$);
\coordinate (S') at ($(v2)!1!60:(S)$);
\draw (v1)--(v2)--(v3)--(v4)--(v5)--cycle;
\draw (v1)--(S)--(v2);
\draw (S)--(S')--(v2);
\node[circle, inner sep = 1.5, fill = white, draw] () at (S) {};
\node[circle, inner sep = 1.5, fill = white, draw] () at (S') {};
\foreach \ang in {1, 2, 3, 4, 5}
{
\node[circle, inner sep = 1.5, fill = white, draw] () at (v\ang) {};
}
\end{tikzpicture}}\hspace{1cm}
\subcaptionbox{\label{fig:subfig:534f}}
{\begin{tikzpicture}[scale = 0.8]
\def\s{1}
\foreach \ang in {1, 2, 3, 4, 5}
{
\def\pointname{v\ang}
\coordinate (\pointname) at ($(\ang*360/5-54:\s)$);
}
\coordinate (S) at ($(v2)!1!60:(v1)$);
\coordinate (S') at ($(v2)!1!90:(S)$);
\coordinate (S'') at ($(S)!1!-90:(v2)$);
\draw (v1)--(v2)--(v3)--(v4)--(v5)--cycle;
\draw (v1)--(S)--(v2);
\draw (S)--(S'')--(S')--(v2);
\node[circle, inner sep = 1.5, fill = white, draw] () at (S) {};
\node[circle, inner sep = 1.5, fill = white, draw] () at (S') {};
\node[circle, inner sep = 1.5, fill = white, draw] () at (S'') {};
\foreach \ang in {1, 2, 3, 4, 5}
{
\node[circle, inner sep = 1.5, fill = white, draw] () at (v\ang) {};
}
\end{tikzpicture}}\hspace{1cm}
\subcaptionbox{\label{fig:subfig:g}}
{\begin{tikzpicture}[scale = 0.8]
\def\s{1.2}
\coordinate (O) at (0, 0);
\coordinate (v1) at (0:\s);
\coordinate (v2) at (60:\s);
\coordinate (v3) at (120:\s);
\coordinate (v4) at (180:\s);
\coordinate (OO) at ($(v3)!1!60:(v2)$);
\coordinate (A) at ($(OO)!1!-60:(v3)$);
\coordinate (B) at ($(OO)!1!-60:(A)$);
\coordinate (C) at ($(OO)!1!-60:(B)$);
\coordinate (D) at ($(OO)!1!-60:(C)$);
\draw (v1)--(v2)--(v3)--(v4)--(O)--cycle;
\draw (O)--(v2);
\draw (O)--(v3);
\draw (v3)--(A)--(B)--(C)--(D)--(v2);
\node[circle, inner sep = 1.5, fill = white, draw] () at (O) {};
\node[circle, inner sep = 1.5, fill = white, draw] () at (v3) {};
\node[circle, inner sep = 1.5, fill = white, draw] () at (v1) {};
\node[circle, inner sep = 1.5, fill = white, draw] () at (v2) {};
\node[circle, inner sep = 1.5, fill = white, draw] () at (v4) {};
\node[circle, inner sep = 1.5, fill = white, draw] () at (A) {};
\node[circle, inner sep = 1.5, fill = white, draw] () at (B) {};
\node[circle, inner sep = 1.5, fill = white, draw] () at (C) {};
\node[circle, inner sep = 1.5, fill = white, draw] () at (D) {};
\end{tikzpicture}}\hspace{1cm}
\subcaptionbox{\label{fig:subfig:h}}
{\begin{tikzpicture}[scale = 0.8]
\def\s{1.2}
\coordinate (O) at (0, 0);
\coordinate (v1) at (0:\s);
\coordinate (v2) at (60:\s);
\coordinate (v3) at (120:\s);
\coordinate (v4) at (180:\s);
\coordinate (OO) at ($(v2)!1!60:(v1)$);
\coordinate (A) at ($(OO)!1!-60:(v2)$);
\coordinate (B) at ($(OO)!1!-60:(A)$);
\coordinate (C) at ($(OO)!1!-60:(B)$);
\coordinate (D) at ($(OO)!1!-60:(C)$);
\draw (v1)--(v2)--(v3)--(v4)--(O)--cycle;
\draw (O)--(v2);
\draw (O)--(v3);
\draw (v2)--(A)--(B)--(C)--(D)--(v1);
\node[circle, inner sep = 1.5, fill = white, draw] () at (O) {};
\node[circle, inner sep = 1.5, fill = white, draw] () at (v3) {};
\node[circle, inner sep = 1.5, fill = white, draw] () at (v1) {};
\node[circle, inner sep = 1.5, fill = white, draw] () at (v2) {};
\node[circle, inner sep = 1.5, fill = white, draw] () at (v4) {};
\node[circle, inner sep = 1.5, fill = white, draw] () at (A) {};
\node[circle, inner sep = 1.5, fill = white, draw] () at (B) {};
\node[circle, inner sep = 1.5, fill = white, draw] () at (C) {};
\node[circle, inner sep = 1.5, fill = white, draw] () at (D) {};
\end{tikzpicture}}
\caption{Forbidden configurations in \cref{PAIRWISE3456}.}
\label{FIGPAIRWISE3456}
\end{figure}
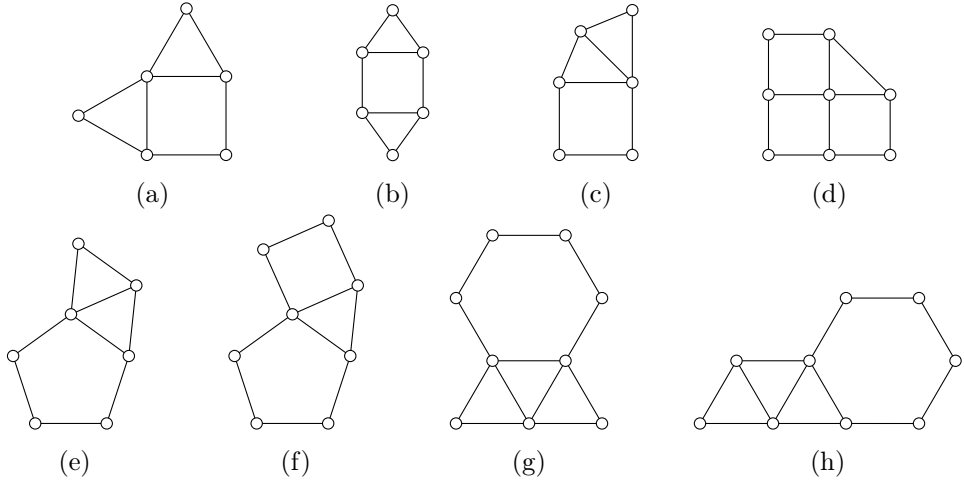

A $4^{-}$-cycle is \textbf{good} if there is no vertex having four neighbors on the cycle, otherwise it is \textbf{a bad cycle}. It is observed that every $3$-cycle is good. Wang \etal \cite{Wang2019+} gave the following result on planar graphs with some restrictions on $6^{-}$-cycles. 

\begin{theorem}[Wang \etal \cite{Wang2019+}]\label{PAIRWISE3456}
Let $G$ be a planar graph without any configuration in \cref{FIGPAIRWISE3456}, and let $x_{1}x_{2}\dots x_{l}x_{1}$ be a good $4^{-}$-cycle in $G$. Let $H$ be a cover of $G$ and $f$ be a function from $V(H)$ to $\{0, 1, 2\}$. If $f(v, 1) + f(v, 2) + \dots + f(v, s) \geq 4$ for each $v \in V(G)$, then each strictly $f$-degenerate transversal $R_{0}$ of $H_{0} = H[\bigcup_{i \in [l]} L_{x_{i}}]$ can be extended to a strictly $f$-degenerate transversal of $H$. 
\end{theorem}

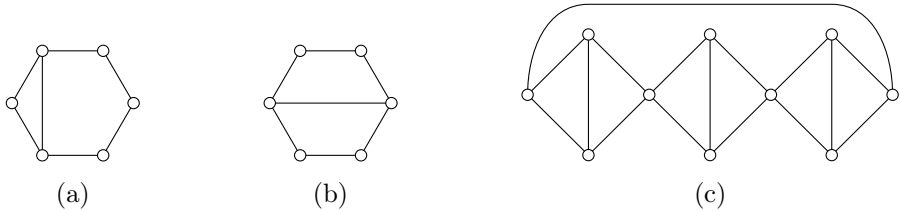
\begin{figure}[htbp]
\centering
\subcaptionbox{}
{\begin{tikzpicture}[scale = 0.8]
\def\s{1}
\foreach \ang in {1, ..., 6}
{
\def\pointname{v\ang}
\coordinate (\pointname) at ($(\ang*60:\s)$);}
\draw (v1)--(v2)--(v3)--(v4)--(v5)--(v6)--cycle;
\draw (v2)--(v4);
\foreach \ang in {1, ..., 6}
{
\node[circle, inner sep = 1.5, fill = white, draw] () at (v\ang) {};
}
\end{tikzpicture}}\hspace{1.5cm}
\subcaptionbox{}
{\begin{tikzpicture}[scale = 0.8]
\def\s{1}
\foreach \ang in {1, ..., 6}
{
\def\pointname{v\ang}
\coordinate (\pointname) at ($(\ang*60:\s)$);}
\draw (v1)--(v2)--(v3)--(v4)--(v5)--(v6)--cycle;
\draw (v3)--(v6);
\foreach \ang in {1, ..., 6}
{
\node[circle, inner sep = 1.5, fill = white, draw] () at (v\ang) {};
}
\end{tikzpicture}}\hspace{1.5cm}
\subcaptionbox{}
{\begin{tikzpicture}[scale = 0.8]
\def\s{1}
\coordinate (E) at (\s, 0);
\coordinate (N) at (0, \s);
\coordinate (W) at (-\s, 0);
\coordinate (S) at (0, -\s);

\coordinate (LN) at (-2*\s, \s);
\coordinate (LW) at (-3*\s, 0);
\coordinate (LS) at (-2*\s, -\s);

\coordinate (RN) at (2*\s, \s);
\coordinate (RE) at (3*\s, 0);
\coordinate (RS) at (2*\s, -\s);

\draw (N)--(W)--(LN)--(LW)--(LS)--(W)--(S)--(E)--(RS)--(RE)--(RN)--(E)--cycle;
\draw (LN)--(LS);
\draw (RN)--(RS);
\draw (N)--(S);
\draw (LW) [out = 90, in = 180] to (-2*\s, 1.5*\s) [out = 0, in = 180] to (2*\s, 1.5*\s) [out = 0, in = 90] to (RE);

\node[circle, inner sep = 1.5, fill = white, draw] () at (E) {};
\node[circle, inner sep = 1.5, fill = white, draw] () at (N) {};
\node[circle, inner sep = 1.5, fill = white, draw] () at (W) {};
\node[circle, inner sep = 1.5, fill = white, draw] () at (S) {};
\node[circle, inner sep = 1.5, fill = white, draw] () at (LN) {};
\node[circle, inner sep = 1.5, fill = white, draw] () at (LW) {};
\node[circle, inner sep = 1.5, fill = white, draw] () at (LS) {};
\node[circle, inner sep = 1.5, fill = white, draw] () at (RN) {};
\node[circle, inner sep = 1.5, fill = white, draw] () at (RE) {};
\node[circle, inner sep = 1.5, fill = white, draw] () at (RS) {};
\end{tikzpicture}}
\caption{Forbidden subgraphs in \cref{Huang}.}
\label{Fig:Huang}
\end{figure}

Huang and Qi \cite{MR4374026} considered a subclass of planar graphs without chorded $6$-cycles, and proved the following result. Note that they did not prove the result for all the planar graphs without chorded $6$-cycles. 

\begin{theorem}[Huang and Qi \cite{MR4374026}]\label{Huang}
Every planar graph without subgraphs isomorphic to the configurations in \cref{Fig:Huang} is DP-$4$-colorable. 
\end{theorem}

In this paper, we mainly prove the following result.

\begin{theorem}\label{C6}
Let $G$ be a planar graph without any subgraph isomorphic to the configurations in \cref{fig:FORBID}, and let $abca$ be a $3$-cycle in $G$. Let $H$ be a cover of $G$ and $f$ be a function from $V(H)$ to $\{0, 1, 2\}$. If $f(v, 1) + f(v, 2) + \dots + f(v, s) \geq 4$ for each $v \in V(G)$, then each strictly $f$-degenerate transversal of $H[L_{a} \cup L_{b} \cup L_{c}]$ can be extended to a strictly $f$-degenerate transversal of $H$. 
\end{theorem}

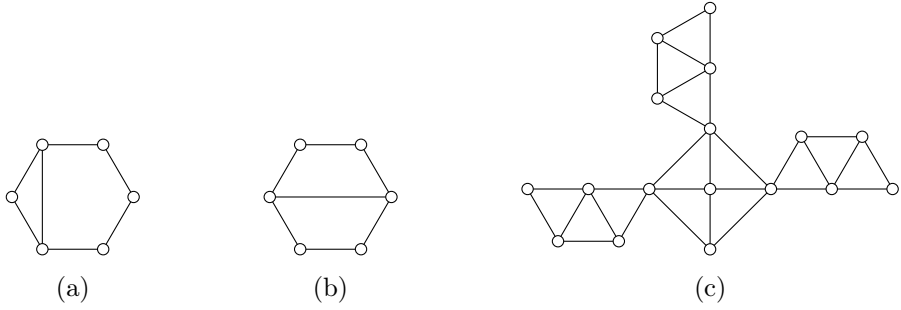
\begin{figure}
\centering
\subcaptionbox{}
{\begin{tikzpicture}[scale = 0.8]
\def\s{1}
\foreach \ang in {1, ..., 6}
{
\def\pointname{v\ang}
\coordinate (\pointname) at ($(\ang*60:\s)$);}
\draw (v1)--(v2)--(v3)--(v4)--(v5)--(v6)--cycle;
\draw (v2)--(v4);
\foreach \ang in {1, ..., 6}
{
\node[circle, inner sep = 1.5, fill = white, draw] () at (v\ang) {};
}
\end{tikzpicture}}\hspace{1.5cm}
\subcaptionbox{}
{\begin{tikzpicture}[scale = 0.8]
\def\s{1}
\foreach \ang in {1, ..., 6}
{
\def\pointname{v\ang}
\coordinate (\pointname) at ($(\ang*60:\s)$);}
\draw (v1)--(v2)--(v3)--(v4)--(v5)--(v6)--cycle;
\draw (v3)--(v6);
\foreach \ang in {1, ..., 6}
{
\node[circle, inner sep = 1.5, fill = white, draw] () at (v\ang) {};
}
\end{tikzpicture}}\hspace{1.5cm}
\subcaptionbox{}
{\begin{tikzpicture}[scale = 0.8]
\def\s{1}
\coordinate (E) at (\s, 0);
\coordinate (N) at (0, \s);
\coordinate (W) at (-\s, 0);
\coordinate (S) at (0, -\s);

\coordinate (EO) at (2*\s, 0);
\coordinate (NO) at (0, 2*\s);
\coordinate (WO) at (-2*\s, 0);

\coordinate (E1) at ($(EO)!1!-60:(E)$);
\coordinate (E2) at ($(EO)!1!-120:(E)$);
\coordinate (E3) at ($(EO)!1!-180:(E)$);
\coordinate (N1) at ($(NO)!1!-60:(N)$);
\coordinate (N2) at ($(NO)!1!-120:(N)$);
\coordinate (N3) at ($(NO)!1!-180:(N)$);
\coordinate (W1) at ($(WO)!1!-60:(W)$);
\coordinate (W2) at ($(WO)!1!-120:(W)$);
\coordinate (W3) at ($(WO)!1!-180:(W)$);

\draw (E)--(N)--(W)--(S)--cycle;
\draw (E3)--(W3);
\draw (S)--(N3);

\draw (E)--(E1)--(E2)--(E3);
\draw (E1)--(EO)--(E2);
\draw (N)--(N1)--(N2)--(N3);
\draw (N1)--(NO)--(N2);
\draw (W)--(W1)--(W2)--(W3);
\draw (W1)--(WO)--(W2);

\foreach \ang in {1, 2, 3}
{
\node[circle, inner sep = 1.5, fill = white, draw] () at (E\ang) {};
\node[circle, inner sep = 1.5, fill = white, draw] () at (N\ang) {};
\node[circle, inner sep = 1.5, fill = white, draw] () at (W\ang) {};
}

\node[circle, inner sep = 1.5, fill = white, draw] () at (E) {};
\node[circle, inner sep = 1.5, fill = white, draw] () at (N) {};
\node[circle, inner sep = 1.5, fill = white, draw] () at (W) {};
\node[circle, inner sep = 1.5, fill = white, draw] () at (S) {};
\node[circle, inner sep = 1.5, fill = white, draw] () at (EO) {};
\node[circle, inner sep = 1.5, fill = white, draw] () at (NO) {};
\node[circle, inner sep = 1.5, fill = white, draw] () at (WO) {};
\node[circle, inner sep = 1.5, fill = white, draw] () at (O) {};
\end{tikzpicture}}

\caption{Forbidden subgraphs in \cref{C6}.}
\label{fig:FORBID}
\end{figure}

Note that a triangle is required in \cref{C6}, but the triangle is not necessary in the following corollary. If $G$ is a triangle-free graph, then we can attach a triangle at a vertex $b$ such that $b$ is a cut vertex, and then by \cref{C6} we obtain a desired strictly $f$-degenerate transversal. 

\begin{corollary}
Let $G$ be a planar graph without any subgraph isomorphic to the configurations in \cref{fig:FORBID}, and let $H$ be a cover of $G$ and $f$ be a function from $V(H)$ to $\{0, 1, 2\}$. If $f(v, 1) + f(v, 2) + \dots + f(v, s) \geq 4$ for each $v \in V(G)$, then $H$ has a strictly $f$-degenerate transversal. 
\end{corollary}

\begin{corollary}
Every planar graph without any subgraph isomorphic to the configurations in \cref{fig:FORBID} is DP-$4$-colorable. 
\end{corollary}

Although it is not proved that every planar graph without chorded $6$-cycles is DP-$4$-colorable, but we believe that it is true, so we make the following conjecture.

\begin{conjecture}
Every planar graph without chorded $6$-cycles is DP-$4$-colorable. 
\end{conjecture}

In all the figures, every solid quadrilateral represents a $4$-vertex, pentagon represents a $5$-vertex, and hexagon represents a $6$-vertex. In addition, a vertex is represented by a red solid point if it is a vertex of outer face.

To finished this section, we introduce some notations. Let $G$ be a plane graph. We use $V(G), E(G)$, and $F(G)$ for the vertex set, the edge set, and the face set respectively. A $k$-vertex ($k^+$-vertex, $k^-$-vertex, respectively) represents a vertex of degree $k$ (at least $k$, at most $k$, respectively). A {\em wheel graph on $n$-vertices} is a graph formed by connecting a single vertex to all vertices of an $(n - 1)$-cycle. 

We use $d(f)$ to denote the degree of a face $f$. Let $b(f)$ be the boundary of a face $f$ and write $f = [v_{1}v_{2}\dots v_{d}]$ when $v_{1}, v_{2}, \dots, v_{d}$ are the boundary vertices of $f$ in a cyclic order. A $k$-face ($k^+$-face, $k^-$-face, respectively) represents a face of degree $k$ (at least $k$, at most $k$, respectively). An {\em internal face} is a bounded face that does not share any vertex with the unbounded face. We use $\int(C)$ and $\ext(C)$ to denote the vertices inside and outside a cycle $C$, respectively. If neither $\int(C)$ nor $\ext(C)$ is empty, then the cycle $C$ is called a {\em separating cycle}.

\section{Preliminaries}

Let $\mathscr{G}$ be a class of graphs which is closed under deleting vertices, \ie every induced subgraph is also in $\mathscr{G}$. 

\begin{theorem}[Theorem 5.3 in Wang \etal \cite{Wang2019+}]\label{WW}
Let $k \geq 3$, and $G$ be a graph in $\mathscr{G}$, and let $K$ be an induced subgraph of $G$ with $V(K) = \{v_{1}, v_{2}, \dots, v_{m}\}$ such that the following hold. 
\begin{enumerate}[label = (\roman*)]
\item\label{WW-1} $k - (d_{G}(v_{1}) - d_{K}(v_{1})) > k - (d_{G}(v_{m}) - d_{K}(v_{m}))$;
\item\label{WW-2} $d_{G}(v_{m}) \leq k$ and $v_{1}v_{m} \in E(G)$; 
\item\label{WW-3} For $2 \leq i \leq m - 1$, $v_{i}$ has at most $k - 1$ neighbors in $G - \{v_{i+1}, \dots, v_{m}\}$.
\end{enumerate}
Let $H$ be a cover of $G$ and $f$ be a function from $V(H)$ to $\{0, 1, 2\}$. If $f(v, 1) + \dots + f(v, s) \geq k$ for each vertex $v \in V(G)$, then any strictly $f$-degenerate transversal of $H - \bigcup_{v \in V(K)}L_{v}$ can be extended to that of $H$. 
\end{theorem}

In the proof of \cref{WW}, a certain color is chosen for $v_{1}$ to ``save'' a color for $v_{m}$. We call $v_{1}$ and $v_{m}$ are {\em paired}. Wang \etal \cite{Wang2019+} generalized \cref{WW} to the following result with at most three nested pairs. 

\begin{theorem}[Theorem 5.5 and 5.6 in Wang \etal \cite{Wang2019+}]\label{EXNN-RC}
Let $G$ be a graph in $\mathscr{G}$, and let $K$ be a subgraph isomorphic to a configuration in \cref{RC-1}--\cref{RC-7} (with the requirements in the captions). Let $H$ be a cover of $G$ and $f$ be a function from $V(H)$ to $\{0, 1, 2\}$. If $f(v, 1) + \dots + f(v, s) \geq 4$ for each vertex $v \in V(G)$, then any strictly $f$-degenerate transversal of $H - \bigcup_{v \in V(K)}L_{v}$ can be extended to that of $H$. 
\end{theorem}

\section{Proof of \cref{C6}}
Assume that $(G, H, f)$ is a counterexample to \cref{C6} with $|V(G)|$ as small as possible, and $R_{0}$ is a strictly $f$-degenerate transversal of $H[L_{a} \cup L_{b} \cup L_{c}]$. Without loss of generality, assume that $G$ has been embedded in the plane. 
\begin{lemma}\label{NS}
There are no separating $3$-cycles. We may assume that $abca$ bounds the outer face $D$. 
\end{lemma}
\begin{proof}
Suppose that there is a separating $3$-cycle $C = xyzx$. By symmetry, we may assume that $\{a, b, c\} \cap \, \int(C) = \emptyset$. By the minimality, $R_{0}$ can be extended to a strictly $f$-degenerate transversal $R_{1}$ of $H[\cup_{v\, \notin\, \mathrm{int}(C)}L_{v}]$. We revise the function $f$ to obtain a function $f^{*}$ on $H[\cup_{v\, \notin\, \mathrm{ext}(C)}L_{v}]$ by $f^{*}(x, i_{x}) = f^{*}(y, i_{y}) = f^{*}(z, i_{z}) = 1$, where $(x, i_{x}), (y, i_{y}), (z, i_{z}) \in R_{1}$. Delete all the edges in $H[L_{a} \cup L_{b} \cup L_{c}]$ and obtain a new cover $H^{*}$. By the minimality, $\{(x, i_{x}), (y, i_{y}), (z, i_{z})\}$ can be extended to a strictly $f^{*}$-degenerate transversal $R^{*}$ of $H^{*}[\cup_{v\, \notin\, \mathrm{ext}(C)}L_{v}]$. Thus, $R_{1} \cup R^{*}$ is a strictly $f$-degenerate transversal of $H$, a contradiction. 

Since there is no separating $3$-cycle, the cycle $abca$ must bound a $3$-face. Without loss of generality, we may assume that $abca$ bounds the outer face $D$. 
\end{proof}

A vertex $v$ in $G$ is {\em internal} if $v$ is not incident with the outer face $D$. A subgraph in $G$ is {\em internal} if all its vertices are internal. 

\begin{lemma}[Theorem 4.4 in Lu \etal \cite{MR4357325}]\label{MINDEG}
Every internal vertex in $G$ has degree at least four. Then there are at most three $3^{-}$-vertices. 
\end{lemma}

Let $G^{*}$ be an auxiliary graph with vertex-set $\{f : \text{$f$ is a $3$-face other than $D$}\}$, two vertices $f_{u}$ and $f_{v}$ in $G^{*}$ being adjacent if and only if the two $3$-faces are adjacent in $G$. A {\em cluster} in $G$ consists of some $3$-faces which corresponds to a connected component of $G^{*}$. By the definition of clusters, every $3$-face other than $D$ belongs to a unique cluster. 

By \cref{NS,MINDEG} and the hypothesis that $G$ has no chorded $6$-cycles, we can obtain the following result on clusters. 
\begin{lemma}
There are four types of clusters as depicted in \cref{Clusters}. 
\end{lemma}

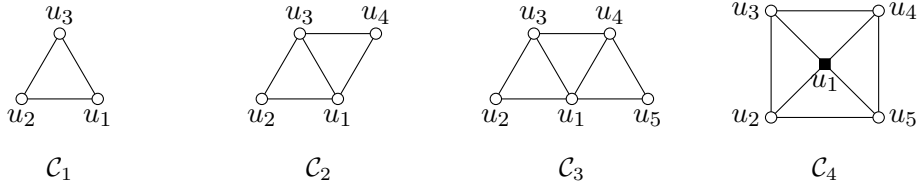
\begin{figure}
\captionsetup[subfigure]{labelformat=empty}
\centering
\subcaptionbox{$\mathcal{C}_{1}$}[0.21\linewidth]{
\begin{tikzpicture}
\def\s{1}
\coordinate (O) at (0, 0);
\coordinate (v1) at (0:\s);
\coordinate (v2) at (60:\s);
\coordinate (v3) at (120:\s);
\coordinate (v4) at (180:\s);
\draw (v3)node[above]{$u_{3}$}--(v4)node[below]{$u_{2}$}--(O)node[below]{$u_{1}$}--cycle;
\node[circle, inner sep = 1.5, fill = white, draw] () at (O) {};
\node[circle, inner sep = 1.5, fill = white, draw] () at (v3) {};
\node[circle, inner sep = 1.5, fill = white, draw] () at (v4) {};
\end{tikzpicture}}
\subcaptionbox{$\mathcal{C}_{2}$}[0.2\linewidth]{
\begin{tikzpicture}
\def\s{1}
\coordinate (O) at (0, 0);
\coordinate (v1) at (0:\s);
\coordinate (v2) at (60:\s);
\coordinate (v3) at (120:\s);
\coordinate (v4) at (180:\s);
\draw (v2)node[above]{$u_{4}$}--(v3)node[above]{$u_{3}$}--(v4)node[below]{$u_{2}$}--(O)node[below]{$u_{1}$}--cycle;
\draw (O)--(v3);
\node[circle, inner sep = 1.5, fill = white, draw] () at (O) {};
\node[circle, inner sep = 1.5, fill = white, draw] () at (v3) {};
\node[circle, inner sep = 1.5, fill = white, draw] () at (v2) {};
\node[circle, inner sep = 1.5, fill = white, draw] () at (v4) {};
\end{tikzpicture}}
\subcaptionbox{$\mathcal{C}_{3}$}[0.2\linewidth]{
\begin{tikzpicture}
\def\s{1}
\coordinate (O) at (0, 0);
\coordinate (v1) at (0:\s);
\coordinate (v2) at (60:\s);
\coordinate (v3) at (120:\s);
\coordinate (v4) at (180:\s);
\draw (v1)node[below]{$u_{5}$}--(v2)node[above]{$u_{4}$}--(v3)node[above]{$u_{3}$}--(v4)node[below]{$u_{2}$}--(O)node[below]{$u_{1}$}--cycle;
\draw (O)--(v2);
\draw (O)--(v3);
\node[circle, inner sep = 1.5, fill = white, draw] () at (O) {};
\node[circle, inner sep = 1.5, fill = white, draw] () at (v3) {};
\node[circle, inner sep = 1.5, fill = white, draw] () at (v1) {};
\node[circle, inner sep = 1.5, fill = white, draw] () at (v2) {};
\node[circle, inner sep = 1.5, fill = white, draw] () at (v4) {};
\end{tikzpicture}}
\subcaptionbox{$\mathcal{C}_{4}$}[0.2\linewidth]{
\begin{tikzpicture}
\def\s{1}
\coordinate (O) at (0, 0);
\coordinate (v1) at (45:\s);
\coordinate (v2) at (135:\s);
\coordinate (v3) at (225:\s);
\coordinate (v4) at (-45:\s);
\draw (v1)node[right]{$u_{4}$}--(v2)node[left]{$u_{3}$}--(v3)node[left]{$u_{2}$}--(v4)node[right]{$u_{5}$}--cycle;
\draw (v1)--(O)node[below]{$u_{1}$}--(v3);
\draw (v2)--(v4);
\node[rectangle, inner sep = 2, fill, draw] () at (O) {};
\node[circle, inner sep = 1.5, fill = white, draw] () at (v3) {};
\node[circle, inner sep = 1.5, fill = white, draw] () at (v1) {};
\node[circle, inner sep = 1.5, fill = white, draw] () at (v2) {};
\node[circle, inner sep = 1.5, fill = white, draw] () at (v4) {};
\end{tikzpicture}}
\caption{Clusters.}
\label{Clusters}
\end{figure}

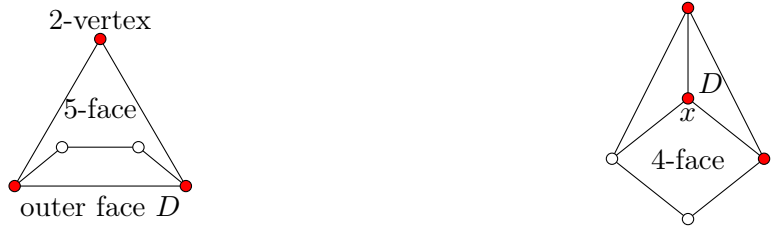
\begin{figure}
\centering
\subcaptionbox{\label{5-3}A $5$-face adjacent to a $3$-face}[0.45\linewidth]{
\begin{tikzpicture}
\def\s{1.3}
\coordinate (x5) at (90: \s);
\coordinate (x1) at (-30: \s);
\coordinate (x2) at (-15: 0.4*\s);
\coordinate (x3) at (195: 0.4*\s);
\coordinate (x4) at (210: \s);
\draw (x1)--(x2)--(x3)--(x4)--(x5)node[above]{2-vertex}--cycle;
\draw (x1)--(x4);
\node[circle, inner sep = 1.5, fill = red, draw] () at (x1) {};
\node[circle, inner sep = 1.5, fill = white, draw] () at (x2) {};
\node[circle, inner sep = 1.5, fill = white, draw] () at (x3) {};
\node[circle, inner sep = 1.5, fill = red, draw] () at (x4) {};
\node[circle, inner sep = 1.5, fill = red, draw] () at (x5) {};
\node () at (0, 0.3*\s) {$5$-face};
\node () at (0, -0.7*\s) {outer face $D$};
\end{tikzpicture}}
\subcaptionbox{\label{4-3-3}A $4$-face adjacent to two $3$-faces, $x$ is a $3$-vertex}[0.5\linewidth]{
\begin{tikzpicture}
\def\s{1}
\coordinate (x1) at (\s, 0);
\coordinate (x2) at (0, 0.8*\s);
\coordinate (x3) at (-\s, 0);
\coordinate (x4) at (0, -0.8*\s);
\coordinate (u) at (0, 2*\s);
\draw (x1)--(x2)node[below]{$x$}--(x3)--(x4)--cycle;
\draw (x1)--(u)--(x3);
\draw (x2)--(u);
\node[circle, inner sep = 1.5, fill = red, draw] () at (x1) {};
\node[circle, inner sep = 1.5, fill = red, draw] () at (x2) {};
\node[circle, inner sep = 1.5, fill = white, draw] () at (x3) {};
\node[circle, inner sep = 1.5, fill = white, draw] () at (x4) {};
\node[circle, inner sep = 1.5, fill = red, draw] () at (u) {};
\node () at (0, 0) {$4$-face};
\node () at ($(x2) + (0.3*\s, 0.2*\s)$) {$D$};
\end{tikzpicture}}
\caption{Certain adjacent faces.}
\end{figure}

When we refer to the clusters, we always use the labels as in \cref{Clusters}. The followings are some structural results on the graph. 
\begin{lemma}\label{ST}
The following statements are true.
\begin{enumerate}[label = (\roman*)]
\item\label{CHORD} If a $4$-cycle $u_{1}u_{2}u_{3}u_{4}u_{1}$ bounds a $4$-face, then the $4$-cycle has no chord. 
\item\label{AA} If a $3$-face is adjacent to a $4^{-}$-face, then they are normally adjacent. 
\item\label{NK4} There is no subgraph isomorphic to $K_{4}$. 
\item There are no adjacent $4$-faces. 
\item\label{N3FACE} Every internal $5^{+}$-vertex $v$ is incident with at most $\lfloor\frac{3d(v)}{4}\rfloor$ triangular faces. 
\item\label{3SIM5} If a $5$-face is adjacent to a $3$-face, then it can only be as depicted in \cref{5-3}. 
\item\label{TFT} If a $4$-face is adjacent to two $3$-faces, then it can only be as depicted in \cref{4-3-3}. Hence, a $4$-face is adjacent to at most two $3$-faces. 
\item\label{C4D} Each face adjacent to $\mathcal{C}_{4}$ is a $6^{+}$-face. 
\item\label{C3D} Each face adjacent to $\mathcal{C}_{3}$ is a $6^{+}$-face or $D$. Moreover, if a $\mathcal{C}_{3}$ is adjacent to $D$, then $u_{2}u_{5} \in E(G)$ and $u_{1}u_{2}u_{5}u_{1}$ bounds $D$. 
\item\label{C2D} Each face adjacent to $\mathcal{C}_{2}$ is a $6^{+}$-face or $D$. 
\item\label{DISTANCE1} For a cluster $\mathcal{C}_{4}$, if $P$ is a path linking $u_{2}$ and $u_{3}$ with $P \cap \{u_{1}, u_{4}, u_{5}\} = \emptyset$, then $|E(P)| \notin \{2, 3, 4\}$. 
\item\label{DISTANCE2} For a cluster $\mathcal{C}_{4}$, if $P$ is a path linking $u_{2}$ and $u_{4}$ with $P \cap \{u_{1}, u_{3}, u_{5}\} = \emptyset$, then $|E(P)| \notin \{1, 2, 3\}$. 
\item\label{C4S1} If a cluster $\mathcal{C}_{4}$ and a cluster $\mathcal{C}_{3}$ have a common $5$-vertex, then there is only one case, as shown in \cref{1CLUSTER}. 
\item\label{C4S2} If a cluster $\mathcal{C}_{4}$ is incident with two clusters $\mathcal{C}_{3}$ and the two common vertices are adjacent $5$-vertices in $\mathcal{C}_{4}$, then there is only one case, as shown in \cref{2CLUSTERS}.
\end{enumerate}
\end{lemma}

\begin{figure}
\centering
\subcaptionbox{\label{1CLUSTER}}
{\begin{tikzpicture}[scale = 0.8]
\def\s{1}
\coordinate (E) at (\s, 0);
\coordinate (N) at (0, \s);
\coordinate (W) at (-\s, 0);
\coordinate (S) at (0, -\s);

\coordinate (EO) at (2*\s, 0);

\coordinate (E1) at ($(EO)!1!-60:(E)$);
\coordinate (E2) at ($(EO)!1!-120:(E)$);
\coordinate (E3) at ($(EO)!1!-180:(E)$);

\draw (E)--(N)--(W)--(S)--cycle;
\draw (E3)--(W);
\draw (S)--(N);

\draw (E)--(E1)node[above left = -1.5]{$x$}--(E2)--(E3);
\draw (E1)--(EO)node[below]{$y$}--(E2);

\foreach \ang in {1, 2, 3}
{
\node[circle, inner sep = 1.5, fill = white, draw] () at (E\ang) {};
}
\node[regular polygon, inner sep = 2, fill = blue, draw = blue] () at (E) {};
\node[circle, inner sep = 1.5, fill = white, draw] () at (N) {};
\node[circle, inner sep = 1.5, fill = white, draw] () at (W) {};
\node[circle, inner sep = 1.5, fill = white, draw] () at (S) {};
\node[circle, inner sep = 1.5, fill = white, draw] () at (EO) {};
\node[circle, inner sep = 1.5, fill = white, draw] () at (O) {};
\end{tikzpicture}}\hspace{1.5cm}
\subcaptionbox{\label{2CLUSTERS}}
{\begin{tikzpicture}[scale = 0.8]
\def\s{1}
\coordinate (E) at (\s, 0);
\coordinate (N) at (0, \s);
\coordinate (W) at (-\s, 0);
\coordinate (S) at (0, -\s);

\coordinate (EO) at (2*\s, 0);
\coordinate (NO) at (0, 2*\s);

\coordinate (E1) at ($(EO)!1!-60:(E)$);
\coordinate (E2) at ($(EO)!1!-120:(E)$);
\coordinate (E3) at ($(EO)!1!-180:(E)$);
\coordinate (N1) at ($(NO)!1!-60:(N)$);
\coordinate (N2) at ($(NO)!1!-120:(N)$);
\coordinate (N3) at ($(NO)!1!-180:(N)$);

\draw (E)--(N)--(W)--(S)--cycle;
\draw (E3)--(W);
\draw (S)--(N3);

\draw (E)--(E1)--(E2)--(E3);
\draw (E1)--(EO)--(E2);
\draw (N)--(N1)--(N2)--(N3);
\draw (N1)--(NO)--(N2);

\foreach \ang in {1, 2, 3}
{
\node[circle, inner sep = 1.5, fill = white, draw] () at (E\ang) {};
\node[circle, inner sep = 1.5, fill = white, draw] () at (N\ang) {};
}

\node[regular polygon, inner sep = 2, fill = blue, draw = blue] () at (E) {};
\node[regular polygon, inner sep = 2, fill = blue, draw = blue] () at (N) {};
\node[circle, inner sep = 1.5, fill = white, draw] () at (W) {};
\node[circle, inner sep = 1.5, fill = white, draw] () at (S) {};
\node[circle, inner sep = 1.5, fill = white, draw] () at (EO) {};
\node[circle, inner sep = 1.5, fill = white, draw] () at (NO) {};
\node[circle, inner sep = 1.5, fill = white, draw] () at (O) {};
\end{tikzpicture}}
\caption{A cluster $\mathcal{C}_{4}$ is incident with some clusters $\mathcal{C}_{3}$.}
\end{figure}
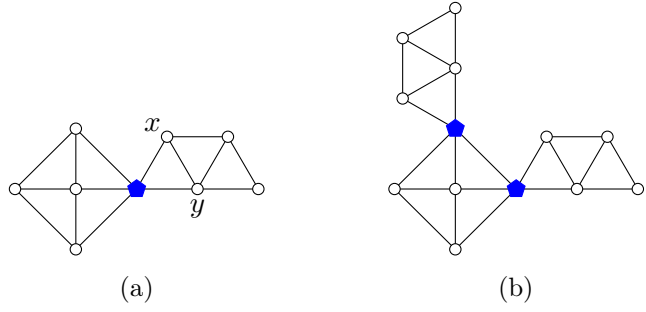

\begin{proof}
\begin{enumerate}[label = (\roman*)]
\item Suppose that $u_{1}$ and $u_{3}$ are adjacent. Since $u_{1}u_{2}u_{3}u_{1}$ is a 3-cycle, it must bound a $3$-face. Thus, $u_{2}$ is a $2$-vertex and $u_{1}u_{2}u_{3}u_{1}$ bounds the outer face. Note that $u_{1}u_{3}u_{4}u_{1}$ is also a $3$-cycle, it bounds a $3$-face, thus $u_{4}$ is a $2$-vertex, but $u_{4}$ should be an internal vertex and has degree at least four, a contradiction.
\item Observe that two adjacent $3$-faces must be normally adjacent. If a $3$-face is adjacent to a $4$-face, then they are normally adjacent because $4$-face has no chord.
\item Assume that there is a subgraph isomorphic to $K_{4}$. One can observe that there are four $3$-cycles in $K_{4}$, and every $3$-cycle bounds a $3$-face. Thus, $G = K_{4}$ which contradicts \cref{MINDEG}. 
\item Assume that a $4$-face $[v_{1}v_{2}v_{3}v_{4}]$ is adjacent to another $4$-face $[v_{1}v_{2}uw]$. By \cref{ST}\ref{CHORD}, we have that $u \neq v_{4}$. If $u = v_{3}$, then $v_{2}$ is a $2$-vertex which is incident with two $4$-faces, contradicts \cref{MINDEG}. Then $u \notin \{v_{3}, v_{4}\}$. By symmetry, we have that $w \notin \{v_{3}, v_{4}\}$. In this case, the two normally adjacent $4$-cycles form a chorded $6$-cycle, a contradiction. 
\item Since there is no chorded $6$-cycle, there is no four consecutive $3$-faces incident with a $5^{+}$-vertex. Then four consecutive faces contain at most three $3$-faces. It follows that every internal $5^{+}$-vertex $v$ is incident with at most $\lfloor\frac{3d(v)}{4}\rfloor$ triangular faces.
\item Assume that a $5$-face $[v_{1}v_{2}v_{3}v_{4}v_{5}]$ is adjacent to a $3$-face $[v_{1}v_{2}u]$. Since there is no chorded $6$-cycle, $u \in \{v_{3}, v_{4}, v_{5}\}$. By symmetry, we need to consider two cases: $u = v_{3}$ or $u = v_{4}$. Suppose that $u = v_{4}$. Then $v_{2}v_{3}v_{4}v_{2}$ bounds a $3$-face, and $v_{3}$ is a $2$-vertex. It follows that $v_{2}v_{3}v_{4}v_{2}$ bounds the outer face $D$, and $v_{5}$ is an internal $4^{+}$-vertex. But $v_{1}v_{4}v_{5}v_{1}$ is a separating $3$-cycle, contradicts \cref{NS}. Hence, $u = v_{3}$ and $v_{2}$ is a $2$-vertex which is incident with the outer $3$-face $D$, see \cref{5-3}. 

\item Let $f = [v_{1}v_{2}v_{3}v_{4}]$ be a $4$-face. Suppose that $f$ is adjacent to two $3$-faces $[v_{1}v_{2}u]$ and $[v_{3}v_{4}w]$. Since $[v_{1}v_{2}v_{3}v_{4}]$ has no chord, we have that $u, w \notin \{v_{1}, v_{2}, v_{3}, v_{4}\}$. Note that $u = w$, for otherwise $v_{1}uv_{2}v_{3}wv_{4}v_{1}$ is a chorded $6$-cycle. By \cref{NS}, every $3$-cycle bounds a $3$-face, thus $u$ is incident with four $3$-faces. Observe that $G$ is a wheel on five vertices, and it has four $3$-vertices, contradicts \cref{MINDEG}. 

Suppose that $f$ is adjacent to two $3$-faces $[v_{1}v_{2}u]$ and $[v_{2}v_{3}z]$. Similarly, we can obtain that $u, z \notin \{v_{1}, v_{2}, v_{3}, v_{4}\}$ and $u = z$. By \cref{NS}, each of $[v_{1}v_{2}u]$ and $[v_{2}v_{3}u]$ bounds a $3$-face. Then $v_{2}$ is a $3$-vertex which is incident with the outer $3$-face $D$, see \cref{4-3-3}. 

By the above discussion, only two consecutive edges can be incident with $3$-faces. Then a $4$-face is adjacent to at most two $3$-faces.

\item Follows from Lemma 2(4) in \cite{MR4374026}. 
\item Follows from Lemma 2(5) in \cite{MR4374026}. 
\item Follows from Lemma 2(5) in \cite{MR4374026}. 
\item Suppose that $|E(P)| \in \{2, 3, 4\}$. Then there is a chorded $6$-cycle, a contradiction. 
\item Suppose that $|E(P)| = 1$. Then $u_{2}u_{4} \in E(G)$ and there is a subgraph isomorphic to $K_{4}$, contradicting \cref{ST}\ref{NK4}. When $|E(P)| \in \{2, 3\}$, there is a chorded $6$-cycle, a contradiction. 
\item By \cref{ST}\ref{C4D} and \ref{NK4}, neither $x$ nor $y$ is on the cluster $\mathcal{C}_{4}$. By \cref{ST}\ref{DISTANCE1} and \ref{DISTANCE2}, $\mathcal{C}_{4}$ and $\mathcal{C}_{3}$ can share exactly one vertex, see \cref{1CLUSTER}. 
\item One can observe that $\mathcal{C}_{4}$ and each cluster $\mathcal{C}_{3}$ share exactly one vertex by \cref{ST}\ref{C4S1}. It suffices to show that the two clusters $\mathcal{C}_{3}$ do not share any vertex. This immediately follows from \cref{ST}\ref{DISTANCE1}. \qedhere
\end{enumerate} 
\end{proof}

The {\em diamond graph} is the simple graph on four vertices and five edges. A subgraph is {\em $4$-regular} if all its vertices have degree four in $G$. 
\begin{lemma}[Theorem 4.5 in Lu et al. \cite{MR4357325}]\label{K4-}
There are no internal $4$-regular diamonds in $G$. 
\end{lemma}

\begin{lemma}\label{ADJ}
For the internal configurations in \cref{fig:subfig:RC-6-b} and \cref{fig:subfig:RC-7-a}, we may always assume that $|N(w_{11}) \cap \{w_{1}, w_{2}, \dots, w_{10}\}| = 2$. 
\end{lemma}
\begin{proof}
Consider the configuration \cref{fig:subfig:RC-6-b}. Since $w_{10}$ has degree $4$ in $G$, we have that either $w_{2}w_{10} \notin E(G)$ or $w_{11}w_{10} \notin E(G)$. By symmetry, we may assume that $w_{10}$ and $w_{11}$ are nonadjacent. Since $G$ has no chorded $6$-cycle and has no subgraph isomorphic to $K_{4}$, it is easy to check that $N(w_{11}) \cap \{w_{1}, w_{2}, \dots, w_{10}\} = \{w_{1}, w_{6}\}$.

Similar arguments can be applied to configuration \cref{fig:subfig:RC-7-a}. 
\end{proof}

We give the following reducible configurations. 

\begin{lemma}\label{R-C}
Every internal configuration in \cref{RC-1}--\cref{RC-7} is reducible. 
\end{lemma}
\begin{proof}
Since $G$ has no chorded $6$-cycle and has no subgraph isomorphic to $K_{4}$, it is easy to check that all the conditions in the captions of \cref{RC-1}, \cref{RC-2}, \cref{RC-3}, and \cref{fig:subfig:RC-6-a} are satisfied. By \cref{EXNN-RC}, they are reducible. 

By \cref{ADJ}, for the internal configurations in \cref{fig:subfig:RC-6-b} and \cref{fig:subfig:RC-7-a}, we may always assume that $|N(w_{11}) \cap \{w_{1}, w_{2}, \dots, w_{10}\}| = 2$. All the other conditions in the captions are also satisfied. By \cref{EXNN-RC}, we have \cref{fig:subfig:RC-6-b} and \cref{fig:subfig:RC-7-a} are reducible. 

Consider internal configuration \cref{fig:subfig:RC-7-b}. If $|N(w_{11}) \cap \{w_{1}, w_{2}, \dots, w_{10}\}| = 2$, then 
\[
w_{1}, w_{2}, \dots, w_{11}
\]
satisfies the condition of \cref{EXNN-RC} with pairs $(w_{1}, w_{11}), (w_{2}, w_{9}), (w_{3}, w_{7})$. If $w_{11}w_{2} \in E(G)$, then 
\[
w_{2}, w_{9}, w_{8}, w_{7}, w_{3}, w_{4}, w_{6}, w_{5}, w_{10}, w_{1}, w_{11}
\]
satisfies the condition of \cref{EXNN-RC} with pair $(w_{2}, w_{11})$. If $w_{11}w_{9} \in E(G)$, then 
\[
w_{6}, w_{4}, w_{5}, w_{11}, w_{2}, w_{7}, w_{8}, w_{9}, w_{10}, w_{1}, w_{3}
\]
satisfies the condition of \cref{EXNN-RC} with pairs $(w_{6}, w_{3}), (w_{11}, w_{1}), (w_{2}, w_{9})$. 
\end{proof}

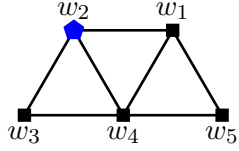
\begin{figure}
\centering
\begin{tikzpicture}[line width = 1pt]
\def\s{1.3}
\coordinate (O) at (0, 0);
\coordinate (v1) at (0:\s);
\coordinate (v2) at (60:\s);
\coordinate (v3) at (120:\s);
\coordinate (v4) at (180:\s);
\draw (v1)node[below]{$w_{5}$}--(v2)node[above]{$w_{1}$}--(v3)node[above]{$w_{2}$}--(v4)node[below]{$w_{3}$}--(O)node[below]{$w_{4}$}--cycle;
\draw (O)--(v2);
\draw (O)--(v3);
\node[rectangle, inner sep = 2, fill, draw] () at (O) {};
\node[regular polygon, inner sep = 2, fill = blue, draw = blue] () at (v3) {};
\node[rectangle, inner sep = 2, fill, draw] () at (v1) {};
\node[rectangle, inner sep = 2, fill, draw] () at (v2) {};
\node[rectangle, inner sep = 2, fill, draw] () at (v4) {};
\end{tikzpicture}
\caption{$w_{2}w_{5}, w_{3}w_{5} \notin E(G)$}
\label{RC-1}
\end{figure}

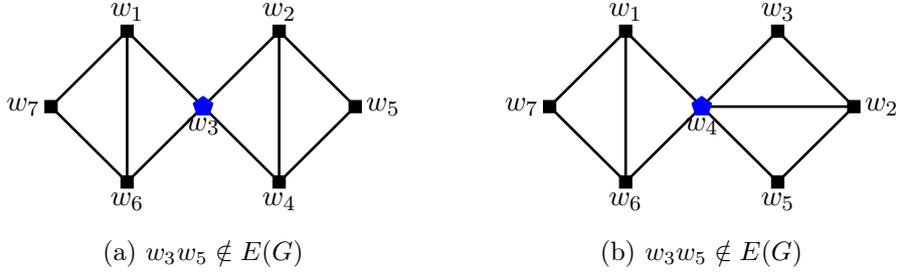
\begin{figure}
\centering
\subcaptionbox{\label{fig:subfig:RC-2-a}$w_{3}w_{5} \notin E(G)$}{\begin{tikzpicture}[line width = 1pt]
\def\s{1}
\coordinate (O) at (0, 0);
\coordinate (v1) at (45: 1.414*\s);
\coordinate (v2) at (135: 1.414*\s);
\coordinate (v3) at (225: 1.414*\s);
\coordinate (v4) at (-45: 1.414*\s);
\coordinate (A) at (-2*\s, 0);
\coordinate (B) at (2*\s, 0);
\draw (v1)node[above]{$w_{2}$}--(O)node[below]{$w_{3}$}--(v3)node[below]{$w_{6}$}--(A)node[left]{$w_{7}$}--(v2)node[above]{$w_{1}$}--(v4)node[below]{$w_{4}$}--(B)node[right]{$w_{5}$}--cycle;
\draw (v2)--(v3);
\draw (v1)--(v4);
\node[rectangle, inner sep = 2, fill, draw] () at (v1) {};
\node[rectangle, inner sep = 2, fill, draw] () at (v2) {};
\node[rectangle, inner sep = 2, fill, draw] () at (v3) {};
\node[rectangle, inner sep = 2, fill, draw] () at (v4) {};
\node[rectangle, inner sep = 2, fill, draw] () at (A) {};
\node[rectangle, inner sep = 2, fill, draw] () at (B) {};
\node[regular polygon, inner sep = 2, fill = blue, draw = blue] () at (O) {};
\end{tikzpicture}}\hspace{1cm}
\subcaptionbox{\label{fig:subfig:RC-2-b}$w_{3}w_{5} \notin E(G)$}{\begin{tikzpicture}[line width = 1pt]
\def\s{1}
\coordinate (O) at (0, 0);
\coordinate (v1) at (45: 1.414*\s);
\coordinate (v2) at (135: 1.414*\s);
\coordinate (v3) at (225: 1.414*\s);
\coordinate (v4) at (-45: 1.414*\s);
\coordinate (A) at (-2*\s, 0);
\coordinate (B) at (2*\s, 0);
\draw (v1)node[above]{$w_{3}$}--(O)node[below]{$w_{4}$}--(v3)node[below]{$w_{6}$}--(A)node[left]{$w_{7}$}--(v2)node[above]{$w_{1}$}--(v4)node[below]{$w_{5}$}--(B)node[right]{$w_{2}$}--cycle;
\draw (v2)--(v3);
\draw (O)--(B);
\node[rectangle, inner sep = 2, fill, draw] () at (v1) {};
\node[rectangle, inner sep = 2, fill, draw] () at (v2) {};
\node[rectangle, inner sep = 2, fill, draw] () at (v3) {};
\node[rectangle, inner sep = 2, fill, draw] () at (v4) {};
\node[rectangle, inner sep = 2, fill, draw] () at (A) {};
\node[rectangle, inner sep = 2, fill, draw] () at (B) {};
\node[regular polygon, inner sep = 2, fill = blue, draw = blue] () at (O) {};
\end{tikzpicture}}
\caption{Note that $|N(w_{7}) \cap \{w_{1}, \dots, w_{6}\}| = 2$.}
\label{RC-2}
\end{figure}

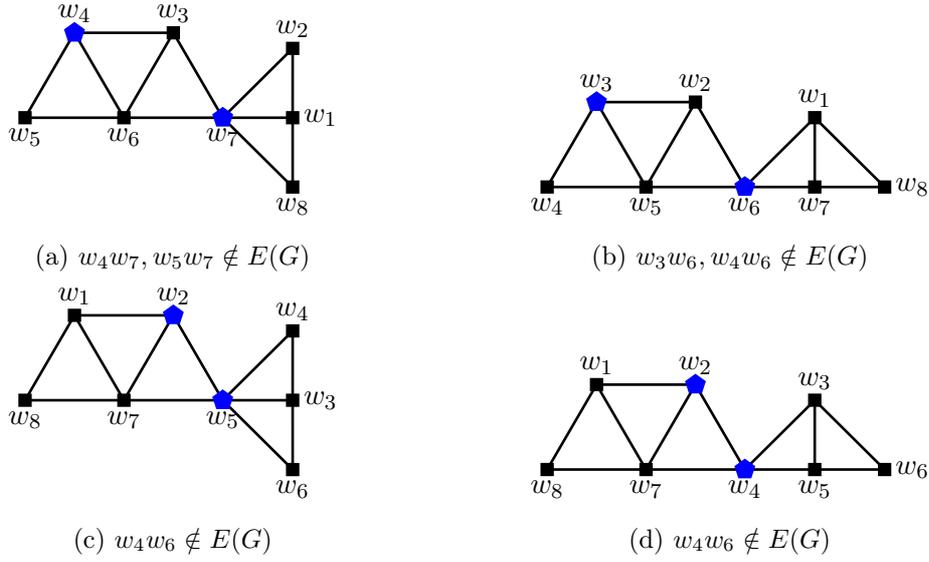
\begin{figure}
\centering
\subcaptionbox{\label{fig:subfig:RC-3-a}$w_{4}w_{7}, w_{5}w_{7} \notin E(G)$}[0.45\linewidth]{
\begin{tikzpicture}[line width = 1pt]
\def\s{1.3}
\coordinate (O) at (0, 0);
\coordinate (v1) at (0:\s);
\coordinate (v2) at (60:\s);
\coordinate (v3) at (120:\s);
\coordinate (v4) at (180:\s);
\coordinate (E1) at ($(45:\s) + (v1)$);
\coordinate (E2) at ($(0:0.707*\s) + (v1)$);
\coordinate (E3) at ($(-45:\s) + (v1)$);
\draw (v1)node[below]{$w_{7}$}--(v2)node[above]{$w_{3}$}--(v3)node[above]{$w_{4}$}--(v4)node[below]{$w_{5}$}--(O)node[below]{$w_{6}$}--cycle;
\draw (O)--(v2);
\draw (O)--(v3);
\draw (E1)node[above]{$w_{2}$}--(E2)node[right]{$w_{1}$}--(E3)node[below]{$w_{8}$}--(v1)--cycle;
\draw (v1)--(E2);
\node[rectangle, inner sep = 2, fill, draw] () at (O) {};
\node[regular polygon, inner sep = 2, fill = blue, draw = blue] () at (v3) {};
\node[regular polygon, inner sep = 2, fill = blue, draw = blue] () at (v1) {};
\node[rectangle, inner sep = 2, fill, draw] () at (v2) {};
\node[rectangle, inner sep = 2, fill, draw] () at (v4) {};
\node[rectangle, inner sep = 2, fill, draw] () at (E1) {};
\node[rectangle, inner sep = 2, fill, draw] () at (E2) {};
\node[rectangle, inner sep = 2, fill, draw] () at (E3) {};
\end{tikzpicture}}
\subcaptionbox{\label{fig:subfig:RC-3-b}$w_{3}w_{6}, w_{4}w_{6} \notin E(G)$}[0.45\linewidth]{
\begin{tikzpicture}[line width = 1pt]
\def\s{1.3}
\coordinate (O) at (0, 0);
\coordinate (v1) at (0:\s);
\coordinate (v2) at (60:\s);
\coordinate (v3) at (120:\s);
\coordinate (v4) at (180:\s);
\coordinate (E1) at ($(45:\s) + (v1)$);
\coordinate (E2) at ($(0:1.414*\s) + (v1)$);
\coordinate (E3) at ($(0:0.707*\s) + (v1)$);
\draw (v1)node[below]{$w_{6}$}--(v2)node[above]{$w_{2}$}--(v3)node[above]{$w_{3}$}--(v4)node[below]{$w_{4}$}--(O)node[below]{$w_{5}$}--cycle;
\draw (O)--(v2);
\draw (O)--(v3);
\draw (E1)node[above]{$w_{1}$}--(E2)node[right]{$w_{8}$}--(E3)node[below]{$w_{7}$}--(v1)--cycle;
\draw (E1)--(E3);
\node[rectangle, inner sep = 2, fill, draw] () at (O) {};
\node[regular polygon, inner sep = 2, fill = blue, draw = blue] () at (v3) {};
\node[regular polygon, inner sep = 2, fill = blue, draw = blue] () at (v1) {};
\node[rectangle, inner sep = 2, fill, draw] () at (v2) {};
\node[rectangle, inner sep = 2, fill, draw] () at (v4) {};
\node[rectangle, inner sep = 2, fill, draw] () at (E1) {};
\node[rectangle, inner sep = 2, fill, draw] () at (E2) {};
\node[rectangle, inner sep = 2, fill, draw] () at (E3) {};
\end{tikzpicture}}
\subcaptionbox{\label{fig:subfig:RC-3-c}$w_{4}w_{6} \notin E(G)$}[0.45\linewidth]{
\begin{tikzpicture}[line width = 1pt]
\def\s{1.3}
\coordinate (O) at (0, 0);
\coordinate (v1) at (0:\s);
\coordinate (v2) at (60:\s);
\coordinate (v3) at (120:\s);
\coordinate (v4) at (180:\s);
\coordinate (E1) at ($(45:\s) + (v1)$);
\coordinate (E2) at ($(0:0.707*\s) + (v1)$);
\coordinate (E3) at ($(-45:\s) + (v1)$);
\draw (v1)node[below]{$w_{5}$}--(v2)node[above]{$w_{2}$}--(v3)node[above]{$w_{1}$}--(v4)node[below]{$w_{8}$}--(O)node[below]{$w_{7}$}--cycle;
\draw (O)--(v2);
\draw (O)--(v3);
\draw (E1)node[above]{$w_{4}$}--(E2)node[right]{$w_{3}$}--(E3)node[below]{$w_{6}$}--(v1)--cycle;
\draw (v1)--(E2);
\node[rectangle, inner sep = 2, fill, draw] () at (O) {};
\node[regular polygon, inner sep = 2, fill = blue, draw = blue] () at (v2) {};
\node[regular polygon, inner sep = 2, fill = blue, draw = blue] () at (v1) {};
\node[rectangle, inner sep = 2, fill, draw] () at (v3) {};
\node[rectangle, inner sep = 2, fill, draw] () at (v4) {};
\node[rectangle, inner sep = 2, fill, draw] () at (E1) {};
\node[rectangle, inner sep = 2, fill, draw] () at (E2) {};
\node[rectangle, inner sep = 2, fill, draw] () at (E3) {};
\end{tikzpicture}}
\subcaptionbox{\label{fig:subfig:RC-3-d}$w_{4}w_{6} \notin E(G)$}[0.45\linewidth]{
\begin{tikzpicture}[line width = 1pt]
\def\s{1.3}
\coordinate (O) at (0, 0);
\coordinate (v1) at (0:\s);
\coordinate (v2) at (60:\s);
\coordinate (v3) at (120:\s);
\coordinate (v4) at (180:\s);
\coordinate (E1) at ($(45:\s) + (v1)$);
\coordinate (E2) at ($(0:1.414*\s) + (v1)$);
\coordinate (E3) at ($(0:0.707*\s) + (v1)$);
\draw (v1)node[below]{$w_{4}$}--(v2)node[above]{$w_{2}$}--(v3)node[above]{$w_{1}$}--(v4)node[below]{$w_{8}$}--(O)node[below]{$w_{7}$}--cycle;
\draw (O)--(v2);
\draw (O)--(v3);
\draw (E1)node[above]{$w_{3}$}--(E2)node[right]{$w_{6}$}--(E3)node[below]{$w_{5}$}--(v1)--cycle;
\draw (E1)--(E3);
\node[rectangle, inner sep = 2, fill, draw] () at (O) {};
\node[regular polygon, inner sep = 2, fill = blue, draw = blue] () at (v2) {};
\node[regular polygon, inner sep = 2, fill = blue, draw = blue] () at (v1) {};
\node[rectangle, inner sep = 2, fill, draw] () at (v3) {};
\node[rectangle, inner sep = 2, fill, draw] () at (v4) {};
\node[rectangle, inner sep = 2, fill, draw] () at (E1) {};
\node[rectangle, inner sep = 2, fill, draw] () at (E2) {};
\node[rectangle, inner sep = 2, fill, draw] () at (E3) {};
\end{tikzpicture}}
\caption{Note that $|N(w_{8}) \cap \{w_{1}, \dots, w_{7}\}| = 2$.}
\label{RC-3}
\end{figure}

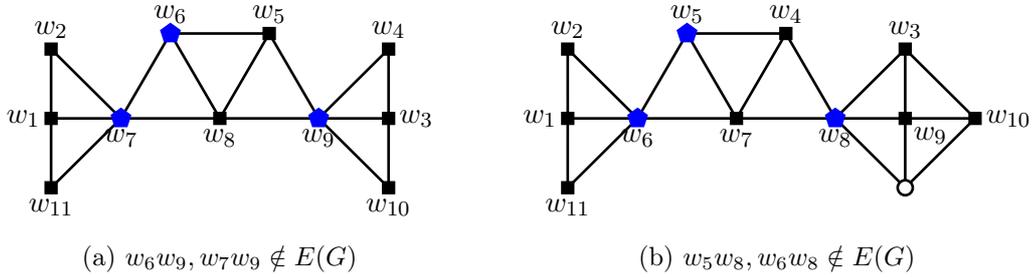
\begin{figure}
\centering
\subcaptionbox{\label{fig:subfig:RC-6-a}$w_{6}w_{9}, w_{7}w_{9} \notin E(G)$}[0.45\linewidth]{
\begin{tikzpicture}[line width = 1pt]
\def\s{1.3}
\coordinate (O) at (0, 0);
\coordinate (v1) at (0:\s);
\coordinate (v2) at (60:\s);
\coordinate (v3) at (120:\s);
\coordinate (v4) at (180:\s);
\coordinate (E1) at ($(45:\s) + (v1)$);
\coordinate (E2) at ($(0:0.707*\s) + (v1)$);
\coordinate (E3) at ($(-45:\s) + (v1)$);
\coordinate (W1) at ($(135:\s) + (v4)$);
\coordinate (W2) at ($(180:0.707*\s) + (v4)$);
\coordinate (W3) at ($(-135:\s) + (v4)$);
\draw (v1)node[below]{$w_{9}$}--(v2)node[above]{$w_{5}$}--(v3)node[above]{$w_{6}$}--(v4)node[below]{$w_{7}$}--(O)node[below]{$w_{8}$}--cycle;
\draw (O)--(v2);
\draw (O)--(v3);
\draw (E1)node[above]{$w_{4}$}--(E2)node[right]{$w_{3}$}--(E3)node[below]{$w_{10}$}--(v1)--cycle;
\draw (v1)--(E2);
\draw (W1)node[above]{$w_{2}$}--(W2)node[left]{$w_{1}$}--(W3)node[below]{$w_{11}$}--(v4)--cycle;
\draw (W2)--(v4);
\node[rectangle, inner sep = 2, fill, draw] () at (O) {};
\node[regular polygon, inner sep = 2, fill = blue, draw = blue] () at (v3) {};
\node[regular polygon, inner sep = 2, fill = blue, draw = blue] () at (v1) {};
\node[rectangle, inner sep = 2, fill, draw] () at (v2) {};
\node[regular polygon, inner sep = 2, fill = blue, draw = blue] () at (v4) {};
\node[rectangle, inner sep = 2, fill, draw] () at (E1) {};
\node[rectangle, inner sep = 2, fill, draw] () at (E2) {};
\node[rectangle, inner sep = 2, fill, draw] () at (E3) {};
\node[rectangle, inner sep = 2, fill, draw] () at (W1) {};
\node[rectangle, inner sep = 2, fill, draw] () at (W2) {};
\node[rectangle, inner sep = 2, fill, draw] () at (W3) {};
\end{tikzpicture}}
\subcaptionbox{\label{fig:subfig:RC-6-b}$w_{5}w_{8}, w_{6}w_{8} \notin E(G)$}[0.45\linewidth]{
\begin{tikzpicture}[line width = 1pt]
\def\s{1.3}
\coordinate (O) at (0, 0);
\coordinate (v1) at (0:\s);
\coordinate (v2) at (60:\s);
\coordinate (v3) at (120:\s);
\coordinate (v4) at (180:\s);
\coordinate (E1) at ($(45:\s) + (v1)$);
\coordinate (E2) at ($(0:1.414*\s) + (v1)$);
\coordinate (E3) at ($(0:0.707*\s) + (v1)$);
\coordinate (E4) at ($(-45:\s) + (v1)$);
\coordinate (W1) at ($(135:\s) + (v4)$);
\coordinate (W2) at ($(180:0.707*\s) + (v4)$);
\coordinate (W3) at ($(-135:\s) + (v4)$);
\draw (v1)node[below]{$w_{8}$}--(v2)node[above]{$w_{4}$}--(v3)node[above]{$w_{5}$}--(v4)node[below]{$w_{6}$}--(O)node[below]{$w_{7}$}--cycle;
\draw (O)--(v2);
\draw (O)--(v3);
\draw (E1)node[above]{$w_{3}$}--(E2)node[right]{$w_{10}$}--(E3)node[below right=-1pt]{$w_{9}$}--(v1)--cycle;
\draw (E1)--(E4);
\draw (E2)--(E4)--(v1);
\draw (W1)node[above]{$w_{2}$}--(W2)node[left]{$w_{1}$}--(W3)node[below]{$w_{11}$}--(v4)--cycle;
\draw (W2)--(v4);
\node[rectangle, inner sep = 2, fill, draw] () at (O) {};
\node[regular polygon, inner sep = 2, fill = blue, draw = blue] () at (v3) {};
\node[regular polygon, inner sep = 2, fill = blue, draw = blue] () at (v1) {};
\node[rectangle, inner sep = 2, fill, draw] () at (v2) {};
\node[regular polygon, inner sep = 2, fill = blue, draw = blue] () at (v4) {};
\node[rectangle, inner sep = 2, fill, draw] () at (E1) {};
\node[rectangle, inner sep = 2, fill, draw] () at (E2) {};
\node[rectangle, inner sep = 2, fill, draw] () at (E3) {};
\node[circle, inner sep =2, fill =white, draw] () at (E4) {};
\node[rectangle, inner sep = 2, fill, draw] () at (W1) {};
\node[rectangle, inner sep = 2, fill, draw] () at (W2) {};
\node[rectangle, inner sep = 2, fill, draw] () at (W3) {};
\end{tikzpicture}}
\caption{Note that $|N(w_{11}) \cap \{w_{1}, \dots, w_{10}\}| = 2$ and $|N(w_{10}) \cap \{w_{3}, \dots, w_{9}\}| = 2$.}
\label{RC-6}
\end{figure}

\begin{figure}
\centering
\subcaptionbox{\label{fig:subfig:RC-7-a}$|N(w_{10}) \cap \{w_{3}, \dots, w_{9}\}| = 2$}[0.45\linewidth]{
\begin{tikzpicture}[line width = 1pt]
\def\s{1.3}
\coordinate (O) at (0, 0);
\coordinate (v1) at (0:\s);
\coordinate (v2) at (60:\s);
\coordinate (v3) at (120:\s);
\coordinate (v4) at (180:\s);
\coordinate (E1) at ($(45:\s) + (v1)$);
\coordinate (E2) at ($(0:0.707*\s) + (v1)$);
\coordinate (E3) at ($(-45:\s) + (v1)$);
\coordinate (W1) at ($(135:\s) + (v4)$);
\coordinate (W2) at ($(180:1.414*\s) + (v4)$);
\coordinate (W3) at ($(180:0.707*\s) + (v4)$);
\coordinate (W4) at ($(-135:\s) + (v4)$);
\draw (v1)node[below]{$w_{10}$}--(v2)node[above]{$w_{3}$}--(v3)node[above]{$w_{4}$}--(v4)node[below]{$w_{6}$}--(O)node[below]{$w_{9}$}--cycle;
\draw (O)--(v2);
\draw (O)--(v3);
\draw (E1)node[above]{$w_{2}$}--(E2)node[right]{$w_{1}$}--(E3)node[right]{$w_{11}$}--(v1)--cycle;
\draw (v1)--(E2);
\draw (W1)node[above]{$w_{5}$}--(W2)node[left]{$w_{8}$}--(W3)node[below left=-1pt]{$w_{7}$}--(v4)--cycle;
\draw (W1)--(W4);
\draw (W2)--(W4)--(v4);
\node[rectangle, inner sep = 2, fill, draw] () at (O) {};
\node[regular polygon, inner sep = 2, fill = blue, draw = blue] () at (v3) {};
\node[regular polygon, inner sep = 2, fill = blue, draw = blue] () at (v1) {};
\node[rectangle, inner sep = 2, fill, draw] () at (v2) {};
\node[regular polygon, inner sep = 2, fill = blue, draw = blue] () at (v4) {};
\node[rectangle, inner sep = 2, fill, draw] () at (E1) {};
\node[rectangle, inner sep = 2, fill, draw] () at (E2) {};
\node[rectangle, inner sep = 2, fill, draw] () at (E3) {};
\node[rectangle, inner sep = 2, fill, draw] () at (W1) {};
\node[rectangle, inner sep = 2, fill, draw] () at (W2) {};
\node[rectangle, inner sep = 2, fill, draw] () at (W3) {};
\node[circle, inner sep =2, fill=white, draw] () at (W4) {};
\end{tikzpicture}}
\subcaptionbox{\label{fig:subfig:RC-7-b}$|N(w_{9}) \cap \{w_{2}, \dots, w_{8}\}| = 2$}[0.45\linewidth]{
\begin{tikzpicture}[line width = 1pt]
\def\s{1.3}
\coordinate (O) at (0, 0);
\coordinate (v1) at (0:\s);
\coordinate (v2) at (60:\s);
\coordinate (v3) at (120:\s);
\coordinate (v4) at (180:\s);
\coordinate (E1) at ($(45:\s) + (v1)$);
\coordinate (E2) at ($(0:1.414*\s) + (v1)$);
\coordinate (E3) at ($(0:0.707*\s) + (v1)$);
\coordinate (E4) at ($(-45:\s) + (v1)$);
\coordinate (W1) at ($(135:\s) + (v4)$);
\coordinate (W2) at ($(180:1.414*\s) + (v4)$);
\coordinate (W3) at ($(180:0.707*\s) + (v4)$);
\coordinate (W4) at ($(-135:\s) + (v4)$);
\draw (v1)node[below]{$w_{7}$}--(v2)node[above]{$w_{3}$}--(v3)node[above]{$w_{4}$}--(v4)node[below]{$w_{5}$}--(O)node[below]{$w_{6}$}--cycle;
\draw (O)--(v2);
\draw (O)--(v3);
\draw (E1)node[above]{$w_{2}$}--(E2)node[right]{$w_{9}$}--(E3)node[below right=-1pt]{$w_{8}$}--(v1)--cycle;
\draw (E1)--(E4);
\draw (E2)--(E4)--(v1);
\draw (W1)node[above]{$w_{1}$}--(W2)node[left]{$w_{11}$}--(W3)node[below left =-1pt]{$w_{10}$}--(v4)--cycle;
\draw (W1)--(W4);
\draw (W2)--(W4)--(v4);
\node[rectangle, inner sep = 2, fill, draw] () at (O) {};
\node[regular polygon, inner sep = 2, fill = blue, draw = blue] () at (v3) {};
\node[regular polygon, inner sep = 2, fill = blue, draw = blue] () at (v1) {};
\node[rectangle, inner sep = 2, fill, draw] () at (v2) {};
\node[regular polygon, inner sep = 2, fill = blue, draw = blue] () at (v4) {};
\node[rectangle, inner sep = 2, fill, draw] () at (E1) {};
\node[rectangle, inner sep = 2, fill, draw] () at (E2) {};
\node[rectangle, inner sep = 2, fill, draw] () at (E3) {};
\node[circle, inner sep =2, fill=white, draw] () at (E4) {};
\node[rectangle, inner sep = 2, fill, draw] () at (W1) {};
\node[rectangle, inner sep = 2, fill, draw] () at (W2) {};
\node[rectangle, inner sep = 2, fill, draw] () at (W3) {};
\node[circle, inner sep =2, fill=white, draw] () at (W4) {};
\end{tikzpicture}}
\caption{Note that $|N(w_{11}) \cap \{w_{1}, \dots, w_{10}\}| = 2$.}
\label{RC-7}
\end{figure}
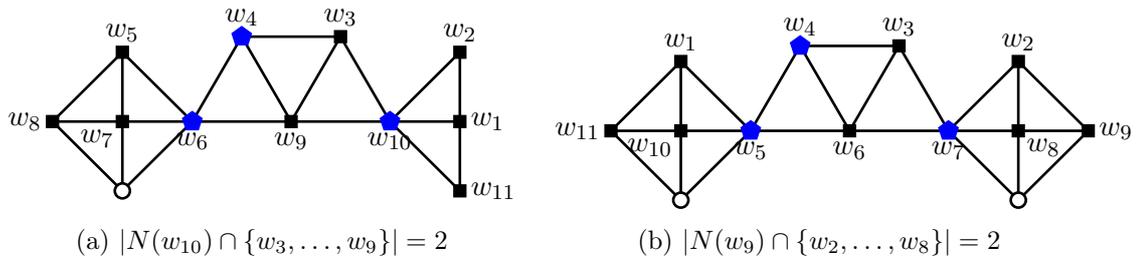

Firstly, we give each element $x \in V(G) \cup F(G) \setminus \{D\}$ an initial charge $\mu(x) = d(x) - 4$, and $\mu(D) = d(D) + 4$. By Euler's formula, the sum of the initial charges is zero, \ie
\begin{equation}
(d(D) + 4) + \sum_{v\, \in\, V(G)}(d(v) - 4) + \sum_{f\, \in\, F(G)\setminus D}(d(f) - 4) = 0.
\end{equation}

Next, we design some discharging rules to redistribute the charges, preserving the sum, such that each vertex and each $4^{+}$-face have nonnegative final charges, and each cluster has nonnegative final charge, while the outer face $D$ has positive final charge, thus the sum of all the final charges is positive, which leads to a contradiction. Let $\mu'$ be the final charge function. For a cluster $\mathcal{C}$, we define the final charge of $\mathcal{C}$ as 
\[
\mu'(\mathcal{C}) = \sum_{f\, \in\, \mathcal{C}} \mu'(f) = \sigma(\mathcal{C}) + \sum_{f\, \in\, \mathcal{C}} \mu(f), 
\]
where $\sigma(\mathcal{C})$ is the sum of the charges obtained by all the members in $\mathcal{C}$. 

An edge is a {\em middle edge} if it is incident with a vertex on the outer cycle but it is not incident with the outer face. Let $v$ be an internal vertex, $t(v)$ be the number of incident $3$-faces. Let $\mu^{*}$ be the charge function after applying the rules from \ref{R1} to \ref{R9}. An internal cluster $\mathcal{C}_{4}$ is {\em special} if it is incident with at least three $5^{+}$-vertices, while an internal cluster $\mathcal{C}_{3}$ is {\em special} if $\mu^{*}(\mathcal{C}_{3}) < 0$. A vertex is called a {\em special vertex} if it is a $5$-vertex which is incident with a special $\mathcal{C}_{3}$ and a special $\mathcal{C}_{4}$. Let $\tau(v \rightarrow f)$ denote the amount of charges transferred from the vertex $v$ to the face $f$. 

\medskip
\noindent\textbf{Discharging rules}: 
\begin{enumerate}[label = \textbf{R\arabic*}, ref = R\arabic*]
\item\label{R1} Let $v$ be an internal $4$-vertex and $f_{1}, f_{2}, f_{3}, f_{4}$ be four consecutive faces incident with $v$. If $d(f_{1}), d(f_{4}) \neq 3$ and $d(f_{2}) = d(f_{3}) = 3$, then $v$ sends $\frac{1}{6}$ to each of $f_{2}$ and $f_{3}$. 
\item\label{R2} Let $v$ be an internal $5$-vertex, and $f$ be an incident $3$-face. Then 
\begin{align*}
\tau(v \rightarrow f) = 
&\begin{cases}
\frac{1}{2}, & \text{if $f$ is in an internal cluster $\mathcal{C}_{4}$;}\\[0.1cm]
0, & \text{if $v$ is in an internal cluster $\mathcal{C}_{4}$ but $f$ is not in the cluster $\mathcal{C}_{4}$;}\\[0.1cm]
\frac{1}{3}, & \text{otherwise.}
\end{cases}
\end{align*}
\item\label{R3} Every internal $6^{+}$-vertex sends $\frac{1}{2}$ to each incident internal $3$-face. 
\item\label{R4} If $f$ is an internal $3$-face and $g$ is an adjacent $4$-face, then $g$ sends $\frac{2}{5}$ to $f$. 
\item\label{R5} If $f$ is a bounded $3$-face and $g$ is an adjacent $m$-face, where $m \geq 6$, then $g$ sends $\frac{m-4}{m}$ to $f$.
\item\label{R6} If $g$ is a $4$-face and $h$ is an adjacent $m$-face, where $m \geq 5$, then $h$ sends $\frac{m-4}{m}$ to $g$.
\item\label{R7} Let $v$ be an internal $4$-vertex. If an edge $uv$ is incident with two $6^{+}$-faces $f$ and $g$, then $f$ sends $\frac{d(f) - 4}{2d(f)}$ to $v$ and $g$ sends $\frac{d(g) - 4}{2d(g)}$ to $v$. 
\item\label{R8} Every vertex on the outer cycle sends its initial charge $\mu(v)$ to the outer face $D$.
\item\label{R9} The outer face $D$ sends $1$ to each middle edge, and every middle edge sends $\frac{1}{2}$ to each incident face. 
\item\label{R10} Let $v$ be a $5$-vertex incident with a special $\mathcal{C}_{4}$. If $v$ is incident with a special $\mathcal{C}_{3}$ with $\mu^{*}(\mathcal{C}_{3}) < - \frac{1}{3}$, then the $\mathcal{C}_{4}$ sends $\frac{1}{3}$ to the $\mathcal{C}_{3}$; if $v$ is incident with a special $\mathcal{C}_{3}$ with $\mu^{*}(\mathcal{C}_{3}) \geq -\frac{1}{3}$, then the $\mathcal{C}_{4}$ sends $-\mu^{*}(\mathcal{C}_{3})$ to the $\mathcal{C}_{3}$. 
\end{enumerate}

\begin{remark}\label{Remark}
Note that each middle edge receives $1$ from $D$, and immediately sends $\frac{1}{2}$ to each incident face. Thus, each middle edge plays a role of agency in the discharging procedure. If $f$ is a face having at least one common vertex with $D$, then it is incident with at least two middle edges, and it receives at least $1$ from $D$. 
\end{remark}

\textbf{Claim 1.} Every vertex in $G$ has nonnegative final charge. 

It is observed that each vertex on the outer cycle has final charge $\mu'(v) = \mu(v) - \mu(v) = 0$ by \ref{R8}. By \cref{MINDEG}, every internal vertex is a $4^{+}$-vertex. Let $v$ be an internal $4$-vertex. If $v$ is incident with at most one $3$-face, then $\mu'(v) = \mu(v) = 0$. If $v$ is incident with exactly two $3$-faces and these two $3$-faces are nonadjacent, then $\mu'(v) = \mu(v) = 0$. If $v$ is incident with at least three $3$-faces, then $\mu'(v) = \mu(v) = 0$. Assume $v$ is incident with exactly two $3$-faces and these two $3$-faces are adjacent. It is observed that the other two faces incident with $v$ are $6^{+}$-faces. By \ref{R7}, $v$ receives at least $\frac{1}{6}$ from each of the incident $6^{+}$-face. By \ref{R1}, $v$ sends $\frac{1}{6}$ to each of incident $3$-face. Thus, $\mu'(v) \geq \mu(v) + 2 \times \frac{1}{6} - 2 \times \frac{1}{6} = 0$. 

If $v$ is an internal $6^{+}$-vertex, then $v$ is incident with at least two $4^{+}$-faces, and then $\mu'(v) \geq \mu(v) - (d(v) - 2) \times \frac{1}{2} = \frac{d(v) - 6}{2} \geq 0$ by \ref{R3}. 

Let $v$ be an internal $5$-vertex incident with five faces $f_{1}, f_{2}, \dots, f_{5}$. By \cref{ST}\ref{N3FACE}, $v$ is incident with at most three $3$-faces. If $v$ is incident with at most two $3$-faces, then $\mu'(v) \geq \mu(v) - 2 \times \frac{1}{2} = 0$ by \ref{R2}. If $v$ is incident with exactly three $3$-faces and it is incident with an internal cluster $\mathcal{C}_{4}$, then $\mu'(v) = \mu(v) - 2 \times \frac{1}{2} = 0$ by \ref{R2}. Otherwise, $v$ sends $\frac{1}{3}$ to each incident $3$-face, then $\mu'(v) = \mu(v) - 3 \times \frac{1}{3} = 0$ by \ref{R2}. 

Therefore, every vertex in $G$ has nonnegative final charge. 

\textbf{Claim 2.} Every face has nonnegative final charge. 

If $f$ is a $4$-face with $V(f) \cap V(D) \neq \emptyset$, then $\mu'(f) \geq \mu(f) + 1 - \frac{2}{5} > 0$ by \cref{Remark}, \ref{R4} and \cref{ST}\ref{TFT}. If $f$ is a $4$-face with $V(f) \cap V(D) = \emptyset$, then $f$ is incident with at most one $3$-face and at least three $5^{+}$-faces, and then $\mu'(f) \geq \mu(f) + 3 \times \frac{1}{5} - \frac{2}{5} > 0$ by \ref{R4} and \ref{R6}. 

If $f$ is a $5$-face, then $f$ is not adjacent to any internal $3$-face, and then $\mu'(f) \geq \mu(f) - 5 \times \frac{1}{5} = 0$ by \ref{R6}. If $f$ is an $m$-face, where $m \geq 6$, then $f$ sends out at most $\frac{m-4}{m}$ via each incident edge to adjacent $4^{-}$-faces or incident $4$-vertices by \ref{R5}, \ref{R6} and \ref{R7}, and then $\mu'(f) \geq \mu(f) - m \times \frac{m-4}{m} = 0$. 

$\bullet$ \textbf{$\mathcal{C}$ is a cluster $\mathcal{C}_{1}$}. If $V(\mathcal{C}) \cap V(D) \neq \emptyset$, then $\mu'(\mathcal{C}) \geq \mu(\mathcal{C}) + 1 = 0$ by \cref{Remark}. Suppose that $\mathcal{C}$ is an internal cluster. Let $t$ be the number of adjacent $4$-faces. By \cref{ST}\ref{3SIM5}, $\mathcal{C}$ is not adjacent to any $5$-face. Then $\mu'(\mathcal{C}) \geq \mu(\mathcal{C}) + t \times \frac{2}{5} + (3 - t) \times \frac{1}{3} = \frac{1}{15}t \geq 0$ by \ref{R4} and \ref{R5}. 

$\bullet$ \textbf{$\mathcal{C}$ is a cluster $\mathcal{C}_{2}$}. By \cref{ST}\ref{C2D}, each face adjacent to $\mathcal{C}_{2}$ is a $6^{+}$-face or $D$. If $u_{1}$ or $u_{3}$ is incident with $D$, then each face in $\mathcal{C}$ receives $1$ from the outer face $D$ by \cref{Remark}, and then $\mu'(\mathcal{C}) \geq \mu(\mathcal{C}) + 2 \times 1 = 0$. If $u_{1}$ and $u_{3}$ are internal vertices, then $\mathcal{C}$ receives at least $\frac{1}{3}$ from each adjacent face by \ref{R5}, and each face in $\mathcal{C}$ receives at least $\frac{1}{6}$ from each of $u_{1}$ and $u_{3}$ by \ref{R1}, \ref{R2} and \ref{R3}, thus $\mu'(\mathcal{C}) \geq \mu(\mathcal{C}) + 4 \times \frac{1}{3} + 4 \times \frac{1}{6} = 0$. 

$\bullet$ \textbf{$\mathcal{C}$ is a cluster $\mathcal{C}_{3}$}. If $u_{1} \in \{a, b, c\}$, then all the faces in $\mathcal{C}$ have common vertices with $D$, and then $\mu'(\mathcal{C}) \geq \mu(\mathcal{C}) + 3 \times 1 = 0$ by \cref{Remark}. Next, assume that $u_{1}$ is an internal vertex. By \cref{ST}\ref{C3D}, each face adjacent to $\mathcal{C}_{3}$ is a $6^{+}$-face, and $|V(\mathcal{C}) \cap V(D)| \leq 1$. In particular, at least one of $u_{3}$ and $u_{4}$ is an internal vertex. By \ref{R1}, \ref{R2} and \ref{R3}, each internal vertex in $\{u_{3}, u_{4}\}$ sends at least $\frac{1}{6}$ to each incident face in $\mathcal{C}$. By \ref{R5}, $\mathcal{C}$ receives at least $\frac{1}{3}$ from each adjacent face. If $u_{1}$ is a $5^{+}$-vertex, then $u_{1}$ sends at least $\frac{1}{3}$ to each face in $\mathcal{C}$ by \ref{R2} and \ref{R3}, thus $\mu'(\mathcal{C}) \geq \mu(\mathcal{C}) + 5 \times \frac{1}{3} + 3 \times \frac{1}{3} + 2 \times \frac{1}{6} = 0$. If $u_{1}$ is an internal $4$-vertex and $|V(\mathcal{C}) \cap V(D)| = 1$, then $\mathcal{C}$ receives at least $1$ from $D$ by \cref{Remark}, and then $\mu'(\mathcal{C}) \geq \mu(\mathcal{C}) + 5 \times \frac{1}{3} + 1 + 2 \times \frac{1}{6} = 0$. 

So we may assume that $\mathcal{C}$ is an internal cluster $\mathcal{C}_{3}$ with $d(u_{1}) = 4$. By symmetry, it suffices to consider the following four cases according to the degrees of $u_{3}$ and $u_{4}$. 

(a) If $d(u_{3}) \geq 5$ and $d(u_{4}) \geq 5$, then each of $u_{3}$ and $u_{4}$ sends at least $\frac{1}{3}$ to each incident face in $\mathcal{C}$, and $\mu'(\mathcal{C}) \geq \mu(\mathcal{C}) + 5 \times \frac{1}{3} + 4 \times \frac{1}{3} = 0$.

(b) Assume $d(u_{3}) = d(u_{4}) = 4$. By \cref{K4-}, $d(u_{2}) \geq 5$ and $d(u_{5}) \geq 5$. Consider the internal vertex $u_{2}$. By \ref{R2} and \ref{R3}, $\tau(u_{2} \rightarrow \mathcal{C}) < \frac{1}{3}$ only if $d(u_{2}) = 5$ and $u_{2}$ is incident with an internal cluster $\mathcal{C}_{4}$. Since the configurations in \cref{RC-2} are forbidden, the internal cluster $\mathcal{C}_{4}$ at $u_{2}$ is a special cluster, thus it sends $\frac{1}{3}$ to the special cluster $\mathcal{C}$ by \ref{R10}. Thus, either $\mathcal{C}$ receives at least $\frac{1}{3}$ from $u_{2}$ or receives $\frac{1}{3}$ from a special cluster $\mathcal{C}_{4}$ incident with $u_{2}$. This is also true for $u_{5}$. Hence, $\mu'(\mathcal{C}) \geq \mu(\mathcal{C}) + 5 \times \frac{1}{3} + 4 \times \frac{1}{6} + 2 \times \frac{1}{3} = 0$. 

(c) Assume $d(u_{3}) \geq 6$ and $d(u_{4}) = 4$. By \ref{R3}, $u_{3}$ sends $\frac{1}{2}$ to each incident face in $\mathcal{C}$. It follows that $\mu'(\mathcal{C}) \geq \mu(\mathcal{C}) + 5 \times \frac{1}{3} + 2 \times \frac{1}{6} + 2 \times \frac{1}{2} = 0$. 

(d) Assume $d(u_{3}) = 5$ and $d(u_{4}) = 4$. By \ref{R2}, $u_{3}$ sends $\frac{1}{3}$ to each incident face in $\mathcal{C}$. It follows that $\mu^{*}(\mathcal{C}) \geq \mu(\mathcal{C}) + 5 \times \frac{1}{3} + 2 \times \frac{1}{6} + 2 \times \frac{1}{3} = - \frac{1}{3}$. Thus, $\mathcal{C}$ needs at most $\frac{1}{3}$ from others. If $u_{2}$ and $u_{5}$ are adjacent, then $u_{1}u_{2}u_{5}u_{1}$ bounds a $3$-face by \cref{NS}, but this contradicts \cref{ST}\ref{C3D}. It follows that $u_{2}$ and $u_{5}$ are nonadjacent. By \cref{ST}\ref{NK4}, $u_{3}$ and $u_{5}$ are also nonadjacent. Since the configuration in \cref{RC-1} is reducible, we have that $d(u_{2}) \geq 5$ or $d(u_{5}) \geq 5$. If $d(u_{2}) \geq 6$ or $d(u_{5}) \geq 6$, then $\mathcal{C}$ receives $\frac{1}{2}$ from incident $6^{+}$-vertex, we are done. So we may assume that $d(u_{2}) \leq 5$ and $d(u_{5}) \leq 5$. There are three subcases to consider. 

(d1) Assume $d(u_{2}) = 4$ and $d(u_{5}) = 5$. By \ref{R2}, $\tau(u_{5} \rightarrow \mathcal{C}) < \frac{1}{3}$ only if $u_{5}$ is incident with an internal cluster $\mathcal{C}_{4}$. Since the configurations in \cref{fig:subfig:RC-3-a} and \cref{fig:subfig:RC-3-b} are reducible, the cluster $\mathcal{C}_{4}$ at $u_{5}$ is special, thus it sends $-\mu^{*}(\mathcal{C})$ to the special cluster $\mathcal{C}$ by \ref{R10}. 

(d2) Assume $d(u_{2}) = 5$ and $d(u_{5}) = 4$. Similar to the above case (d1), either $\mathcal{C}$ receives $\frac{1}{3}$ from $u_{2}$ or receives $-\mu^{*}(\mathcal{C})$ from a special cluster $\mathcal{C}_{4}$ at $u_{2}$, we are done (here the configurations in \cref{fig:subfig:RC-3-c} and \cref{fig:subfig:RC-3-d} are reducible). 

(d3) Assume $d(u_{2}) = d(u_{5}) = 5$. Assume that $\tau(u_{2} \rightarrow \mathcal{C}) = \tau(u_{5} \rightarrow \mathcal{C}) = 0$, \ie each of $u_{2}$ and $u_{5}$ is incident with an internal cluster $\mathcal{C}_{4}$. Since the configurations in \cref{RC-6} and \cref{RC-7} are reducible, either the cluster $\mathcal{C}_{4}$ at $u_{2}$ or $u_{5}$ is a special cluster. By \ref{R10}, $\mathcal{C}$ receives $-\mu^{*}(\mathcal{C})$ from incident special cluster $\mathcal{C}_{4}$. So we may assume that $\tau(u_{2} \rightarrow \mathcal{C}) + \tau(u_{5} \rightarrow \mathcal{C}) \geq \frac{1}{3}$, then $\mu'(\mathcal{C}) \geq \mu(\mathcal{C}) + 5 \times \frac{1}{3} + 2 \times \frac{1}{6} + 2 \times \frac{1}{3} + \frac{1}{3} = 0$, we are done. 

By the above discussion, if $\mathcal{C}_{3}$ is special, then $d(u_{1}) = d(u_{3}) = d(u_{4}) = 4$ or $d(u_{1}) = d(u_{3}) - 1 = d(u_{4}) = 4$. 

$\bullet$ \textbf{$\mathcal{C}$ is a cluster $\mathcal{C}_{4}$}. Note that $u_{1}$ is a $4$-vertex and it is incident with four faces in $\mathcal{C}$, so $u_{1}$ is an internal vertex. By \cref{ST}\ref{NK4} and \ref{C4D}, $|V(\mathcal{C}) \cap V(D)| \leq 1$. By \cref{ST}\ref{C4D}, $\mathcal{C}$ is adjacent to four $6^{+}$-faces. By \ref{R5}, $\mathcal{C}$ receives at least four $\frac{1}{3}$ from adjacent faces. Assume $|V(\mathcal{C}) \cap V(D)| = 1$. Then $D$ sends two $1$ to the faces in $\mathcal{C}$ by \cref{Remark}. It follows that $\mu'(\mathcal{C}) \geq \mu(\mathcal{C}) + 4 \times \frac{1}{3} + 2 \times 1 + 6 \times \frac{1}{6} > 0$. So we may assume that $\mathcal{C}$ is an internal cluster. By \cref{K4-}, at least two vertices in $\{u_{2}, u_{3}, u_{4}, u_{5}\}$ are $5^{+}$-vertices. If exactly two of $\{u_{2}, u_{3}, u_{4}, u_{5}\}$ are $4$-vertices, then $\mu'(\mathcal{C}) \geq \mu(\mathcal{C}) + 4 \times \frac{1}{3} + 4 \times \frac{1}{2} + 4 \times \frac{1}{6} = 0$. If $\mathcal{C}$ is incident with four $5^{+}$-vertices, then $\mu'(\mathcal{C}) \geq \mu(\mathcal{C}) + 4 \times \frac{1}{3} + 8 \times \frac{1}{2} - 4 \times \frac{1}{3} = 0$. Assume $\mathcal{C}$ is incident with exactly three $5^{+}$-vertices. If $\mathcal{C}$ is incident with at most two special vertices, then $\mu'(\mathcal{C}) \geq \mu(\mathcal{C}) + 4 \times \frac{1}{3} + 6 \times \frac{1}{2} + 2 \times \frac{1}{6} - 2 \times \frac{1}{3} = 0$. 

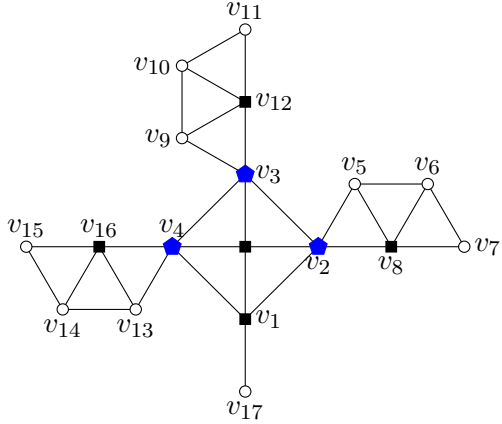
\begin{figure}
\centering
\begin{tikzpicture}[scale = 0.8]
\def\s{1.2}
\coordinate (E) at (\s, 0);
\coordinate (N) at (0, \s);
\coordinate (W) at (-\s, 0);
\coordinate (S) at (0, -\s);

\coordinate (EO) at (2*\s, 0);
\coordinate (NO) at (0, 2*\s);
\coordinate (WO) at (-2*\s, 0);
\coordinate (SO) at (0, -2*\s);

\coordinate (E1) at ($(EO)!1!-60:(E)$);
\coordinate (E2) at ($(EO)!1!-120:(E)$);
\coordinate (E3) at ($(EO)!1!-180:(E)$);
\coordinate (N1) at ($(NO)!1!-60:(N)$);
\coordinate (N2) at ($(NO)!1!-120:(N)$);
\coordinate (N3) at ($(NO)!1!-180:(N)$);
\coordinate (W1) at ($(WO)!1!-60:(W)$);
\coordinate (W2) at ($(WO)!1!-120:(W)$);
\coordinate (W3) at ($(WO)!1!-180:(W)$);

\draw (E)node[below]{$v_{2}$}--(N)node[right]{$v_{3}$}--(W)node[above]{$v_{4}$}--(S)node[right]{$v_{1}$}--cycle;
\draw (E3)--(W3);
\draw (SO)node[below]{$v_{17}$}--(N3);

\draw (E)--(E1)node[above]{$v_{5}$}--(E2)node[above]{$v_{6}$}--(E3)node[right]{$v_{7}$};
\draw (E1)--(EO)node[below]{$v_{8}$}--(E2);
\draw (N)--(N1)node[left]{$v_{9}$}--(N2)node[left]{$v_{10}$}--(N3)node[above]{$v_{11}$};
\draw (N1)--(NO)node[right]{$v_{12}$}--(N2);
\draw (W)--(W1)node[below]{$v_{13}$}--(W2)node[below]{$v_{14}$}--(W3)node[above]{$v_{15}$};
\draw (W1)--(WO)node[above]{$v_{16}$}--(W2);

\foreach \ang in {1, 2, 3}
{
\node[circle, inner sep = 1.5, fill = white, draw] () at (E\ang) {};
\node[circle, inner sep = 1.5, fill = white, draw] () at (N\ang) {};
\node[circle, inner sep = 1.5, fill = white, draw] () at (W\ang) {};
}

\node[regular polygon, inner sep = 2, fill = blue, draw = blue] () at (E) {};
\node[regular polygon, inner sep = 2, fill = blue, draw = blue] () at (N) {};
\node[regular polygon, inner sep = 2, fill = blue, draw = blue] () at (W) {};
\node[rectangle, inner sep = 2, fill, draw] () at (S) {};
\node[rectangle, inner sep = 2, fill, draw] () at (EO) {};
\node[rectangle, inner sep = 2, fill, draw] () at (NO) {};
\node[rectangle, inner sep = 2, fill, draw] () at (WO) {};
\node[circle, inner sep = 1.5, fill = white, draw] () at (SO) {};
\node[rectangle, inner sep = 2, fill, draw] () at (O) {};
\end{tikzpicture}
\caption{A special cluster $C_{4}$ is incident with three special vertices.}
\label{SPECIALC4}
\end{figure}

Next, we prove the statement: \textsl{if $\mathcal{C}$ is incident with exactly three special vertices, then $\mu'(\mathcal{C}) \geq 0$}.

Assume that $\mathcal{C}$ is incident with exactly three special vertices $v_{2}$, $v_{3}$ and $v_{4}$ in \cref{SPECIALC4}. Note that each of $v_{1}, v_{8}, v_{12}$ and $v_{16}$ is a $4$-vertex. By \cref{ST}\ref{C4D}, the special cluster $\mathcal{C}$ is adjacent to four $6^{+}$-faces. By \ref{R5}, $\mathcal{C}$ receives at least four $\frac{1}{3}$ from adjacent faces. Since $\mathcal{C}$ is a special cluster, it receives six $\frac{1}{2}$ from the three incident $5^+$-vertices, and two $\frac{1}{6}$ from the other incident $4$-vertex due to \ref{R1}, \ref{R2} and \ref{R3}. Thus, $\mu^{*}(\mathcal{C}) \geq -4 + 4 \times \frac{1}{3} + 6 \times \frac{1}{2} + 2 \times \frac{1}{6} = \frac{2}{3}$. 

By the assumption of \cref{C6}, the three special clusters $\mathcal{C}_{3}$ at $v_{2}, v_{3}$ and $v_{4}$ must have some common vertices. By \cref{ST}\ref{DISTANCE2}, \ref{C4S1} and \ref{C4S2}, $v_{17}$ has no neighbor in $\{v_{2}, v_{4}, v_{5}, v_{6}, v_{7}, v_{8}, v_{13}, v_{14}, v_{15}, v_{16}\}$, and $\{v_{6}, v_{7}\} \cap \{v_{14}, v_{15}\} \neq \emptyset$. Moreover, there is no edge between $\{v_{9}, v_{12}\}$ and $\{v_{14}, v_{15}\}$. By \cref{ST}\ref{C3D}, it is observed that $|\{v_{6}, v_{7}\} \cap \{v_{14}, v_{15}\}| = 1$. Since $v_{2}$ and $v_{4}$ are special vertices, we have that $v_{5}, v_{6}, v_{13}$ and $v_{14}$ are $5^{-}$-vertices. Then $v_{6}$ and $v_{14}$ can not be identical. By symmetry, we only need to consider two cases. 

\textbf{Case 1: $v_{15}$ and $v_{7}$ are identical.} Assume that $v_{7}$ is a $5$-vertex. Then $v_{7}$ sends $\frac{1}{3}$ to each incident $3$-face. Since the configuration in \cref{fig:subfig:RC-2-a} is reducible, we have that one vertex in $\{v_{5}, v_{6}, v_{13}, v_{14}\}$ has degree $5$. For this special cluster $\mathcal{C}_{3}$, we have $\mu^{*}(\mathcal{C}_{3}) \geq \mu(\mathcal{C}_{3}) + 5 \times \frac{1}{3} + 3 \times \frac{1}{3} + 2 \times \frac{1}{6} = 0$, but it contradicts the definition of special $\mathcal{C}_{3}$. 

Assume that $v_{7}$ is a $6^{+}$-vertex. Then $v_{7}$ sends $\frac{1}{2}$ to each incident $3$-face. The special clusters $\mathcal{C}_{3}$ at $v_{2}$ and $v_{4}$ have $\mu^{*}(\mathcal{C}_{3}) \geq -3 + 5 \times \frac{1}{3} + 4 \times \frac{1}{6} + \frac{1}{2} = - \frac{1}{6}$. Therefore, $\mathcal{C}$ sends at most $\frac{1}{6}$ via each of $v_{2}$ and $v_{4}$, and sends at most $\frac{1}{3}$ via $v_{3}$, so $\mu'(\mathcal{C}) \geq \mu^{*}(\mathcal{C}) - 2 \times \frac{1}{6} - \frac{1}{3} = 0$. 

Assume that $v_{7}$ is a $4$-vertex, \ie it has only four neighbors $v_{6}, v_{8}, v_{14}$ and $v_{16}$. By \cref{K4-}, either $v_{5}$ or $v_{6}$ is a $5$-vertex, and either $v_{13}$ or $v_{14}$ is a $5$-vertex. If the path $v_{7}v_{8}v_{2}v_{1}v_{17}$ is on the boundary of a face, then the face has degree at least $7$ because $v_{17}v_{14}, v_{17}v_{16} \notin E(G)$. If the path $v_{7}v_{8}v_{2}v_{3}$ is on the boundary of a face, then the face has degree at least $7$ because neither $v_{9}$ nor $v_{12}$ is adjacent to $\{v_{14}, v_{16}\}$. Hence, in the two situations, $v_{7}v_{8}, v_{8}v_{2}$, one of $v_{2}v_{1}$ and $v_{2}v_{3}$, one of $v_{15}v_{14}$ and $v_{15}v_{16}$ are incident with $7^{+}$-faces. Note that a $7^{+}$-face sends more $\frac{3}{7} - \frac{1}{3}$ to the special $\mathcal{C}$ than a $6$-face. Therefore, $\mu'(\mathcal{C}) \geq \frac{2}{3} - 3 \times \frac{1}{3} + 4 \times (\frac{3}{7} - \frac{1}{3}) > 0$. 

\textbf{Case 2: $v_{14}$ and $v_{7}$ are identical.} Then $v_{14}$ is a $5$-vertex. 
Similar to the above case when $v_{7}$ and $v_{15}$ are identical as a $4$-vertex, $v_{7}v_{8}, v_{8}v_{2}$, one of $v_{2}v_{1}$ and $v_{2}v_{3}$, one of $v_{14}v_{15}$ and $v_{14}v_{13}$ are incident with $7^{+}$-faces. Therefore, $\mu'(\mathcal{C}) \geq \frac{2}{3} - 3 \times \frac{1}{3} + 4 \times (\frac{3}{7} - \frac{1}{3}) > 0$. This completes the proof of the statement.

Finally, we consider the final charge of the outer face $D$. Note that each vertex $v$ on $D$ is incident with $d(v) - 2$ middle edges. 
\begin{equation*}
\mu'(D) = d(D) + 4 + \sum_{v \in V(D)} \big(d(v) - 4\big) - \sum_{v \in V(D)} \big(d(v) - 2\big) = 4 - d(D) = 1. 
\end{equation*}

\vskip 0mm \vspace{0.3cm} \noindent{\bf Acknowledgments.} The authors would like to thank the anonymous reviewers for their valuable comments and careful reading of this paper. Danjun Huang was supported by NSFC (Nos: 12171436, 12371360), and Weifan Wang was supported by NSFC (No. 12031018) and NSFSD (No. ZR2022MA060).


\begin{thebibliography}{10}

\bibitem{MR3996735}
L.~Chen, R.~Liu, G.~Yu, R.~Zhao and X.~Zhou, {DP}-4-colorability of two classes
  of planar graphs, Discrete Math. 342~(11) (2019) 2984--2993.

\bibitem{MR3758240}
Z.~Dvo\v{r}\'{a}k and L.~Postle, Correspondence coloring and its application to
  list-coloring planar graphs without cycles of lengths 4 to 8, J. Combin.
  Theory Ser. B 129 (2018) 38--54.

\bibitem{MR593902}
P.~Erd\H{o}s, A.~L. Rubin and H.~Taylor, Choosability in graphs, in:
  Proceedings of the {W}est {C}oast {C}onference on {C}ombinatorics, {G}raph
  {T}heory and {C}omputing ({H}umboldt {S}tate {U}niv., {A}rcata, {C}alif.,
  1979), Congress. Numer., XXVI, Utilitas Math., Winnipeg, Man., 1980, pp.
  125--157.

\bibitem{MR2538645}
B.~Farzad, Planar graphs without 7-cycles are 4-choosable, SIAM J. Discrete
  Math. 23~(3) (2009) 1179--1199.

\bibitem{MR4212281}
D.~Huang and J.~Qi, D{P}-coloring on planar graphs without given adjacent short
  cycles, Discrete Math. Algorithms Appl. 13~(2) (2021) 2150013.

\bibitem{MR4374026}
D.~Huang and J.~Qi, Planar graphs without chordal 6-cycles and necklaces are
  {DP}-4-colorable, Adv. Math. (China) 50~(6) (2021) 861--876.

\bibitem{MR3802151}
S.-J. Kim and K.~Ozeki, A sufficient condition for {DP}-4-colorability,
  Discrete Math. 341~(7) (2018) 1983--1986.

\bibitem{MR3969022}
S.-J. Kim and X.~Yu, Planar graphs without 4-cycles adjacent to triangles are
  {DP}-4-colorable, Graphs Combin. 35~(3) (2019) 707--718.

\bibitem{MR1842116}
P.~C.~B. Lam, W.~C. Shiu and B.~Xu, On structure of some plane graphs with
  application to choosability, J. Combin. Theory Ser. B 82~(2) (2001) 285--296.

\bibitem{MR1687318}
P.~C.~B. Lam, B.~Xu and J.~Liu, The 4-choosability of plane graphs without
  4-cycles, J. Combin. Theory Ser. B 76~(1) (1999) 117--126.

\bibitem{MR4294211}
R.~Li and T.~Wang, D{P}-4-coloring of planar graphs with some restrictions on
  cycles, Discrete Math. 344~(11) (2021) 112568.

\bibitem{MR4669976}
R.~Li and T.~Wang, Variable degeneracy on toroidal graphs, Graphs Combin.
  39~(6) (2023) 127.

\bibitem{MR4654340}
X.~Li and M.~Zhang, Every planar graph without 5-cycles adjacent to 6-cycles is
  {DP}-4-colorable, Australas. J. Combin. 87 (2023) 86--97.

\bibitem{MR3881665}
R.~Liu and X.~Li, Every planar graph without 4-cycles adjacent to two triangles
  is {DP}-4-colorable, Discrete Math. 342~(3) (2019) 623--627.

\bibitem{MR4078909}
R.~Liu, X.~Li, K.~Nakprasit, P.~Sittitrai and G.~Yu, {DP}-4-colorability of
  planar graphs without adjacent cycles of given length, Discrete Appl. Math.
  277 (2020) 245--251.

\bibitem{MR4357325}
F.~Lu, Q.~Wang and T.~Wang, Cover and variable degeneracy, Discrete Math.
  345~(4) (2022) 112765.

\bibitem{MR4114324}
K.~M. Nakprasit and K.~Nakprasit, A generalization of some results on list
  coloring and {DP}-coloring, Graphs Combin. 36~(4) (2020) 1189--1201.

\bibitem{MR4089638}
P.~Sittitrai and K.~Nakprasit, Every planar graph without pairwise adjacent 3-,
  4-, and 5-cycle is {DP}-4-colorable, Bull. Malays. Math. Sci. Soc. 43~(3)
  (2020) 2271--2285.

\bibitem{MR4345150}
P.~Sittitrai and K.~Nakprasit, An analogue of {DP}-coloring for variable
  degeneracy and its applications, Discuss. Math. Graph Theory 42~(1) (2022)
  89--99.

\bibitem{MR1290638}
C.~Thomassen, Every planar graph is $5$-choosable, J. Combin. Theory Ser. B
  62~(1) (1994) 180--181.

\bibitem{MR0498216}
V.~G. Vizing, Coloring the vertices of a graph in prescribed colors, Diskret.
  Analiz. 29 (1976) 3--10.

\bibitem{MR1235909}
M.~Voigt, List colourings of planar graphs, Discrete Math. 120~(1-3) (1993)
  215--219.

\bibitem{Wang2019+}
Q.~Wang, T.~Wang and X.~Yang, Variable degeneracy of graphs with restricted
  structures, arXiv:2112.09334,
  \url{https://doi.org/10.48550/arXiv.2112.09334}.

\bibitem{MR1889505}
W.~Wang and K.-W. Lih, Choosability and edge choosability of planar graphs
  without five cycles, Appl. Math. Lett. 15~(5) (2002) 561--565.

\bibitem{MR1935837}
W.~Wang and K.-W. Lih, Choosability and edge choosability of planar graphs
  without intersecting triangles, SIAM J. Discrete Math. 15~(4) (2002)
  538--545.

\end{thebibliography}
\end{document}